\documentclass[12pt,a4paper]{amsproc}
\usepackage[english]{babel}
\usepackage{latexsym}
\usepackage{amsmath,enumitem,amssymb}
\usepackage{amscd}
\usepackage{xcolor}
\usepackage{todonotes}
\usepackage{marginnote}
\usepackage[colorlinks=false]{hyperref}
\usepackage[cp850]{inputenc}
\usepackage[mathscr]{eucal}
\theoremstyle{plain}

\newtheorem{quest}{Question}

\newtheorem{defin}[quest]{Definition}
\newtheorem{theorem}[quest]{Theorem}
\newtheorem{prop}[quest]{Proposition}
\newtheorem{corollary}[quest]{Corollary}
\newtheorem{lemma}[quest]{Lemma}

\newtheorem*{prop*}{Proposition}
\newtheorem*{defin*}{Definition}
\theoremstyle{remark}
\newtheorem{example}[quest]{Example}
\newtheorem{remark}[quest]{Remark}

\newcommand{\adef}{\begin{defin}}
\newcommand{\zdef}{\end{defin}}

\newcommand{\aproof}{\begin{proof}}
\newcommand{\zproof}{\end{proof}}

\begin{document}

\title[On measurability of Kurzweil--Stieltjes integrable functions on compact lines]
{On measurability of Kurzweil--Stieltjes integrable functions on compact lines}

\author{Leandro Candido}
\address{Universidade Federal de S\~ao Paulo - UNIFESP. Instituto de Ci\^encia e Tecnologia. Departamento de Matem\'atica. S\~ao Jos\'e dos Campos - SP, Brasil}
\email{\texttt{leandro.candido@unifesp.br}}
\author{Pedro L. Kaufmann}
\address{Universidade Federal de S\~ao Paulo - UNIFESP. Instituto de Ci\^encia e Tecnologia. Departamento de Matem\'atica. S\~ao Jos\'e dos Campos - SP, Brasil}
\email{\texttt{plkaufmann@unifesp.br}}

\thanks{ The authors are supported by Funda\c c\~ao de Amparo \`a Pesquisa do Estado de S\~ao Paulo - FAPESP No. 2023/12916-1 }

%    General info


\maketitle
\begin{abstract}
We continue the study on Kurzweil--Stieltjes integration on compact lines initiated in \cite{CanKau}. Given a real valued function $G$ on a compact line, the presented integral is called the Kurzweil--Stieltjes integral with respect to $G$, or simply the $G$-integral.  %Given a compact line $K$ and a right-continuous function $G:K\to\mathbb{R}$ of bounded variation, we consider the Radon measure $\mu_G$ naturally induced by $G$.
Our main results concern the relationship between $G$-integrability and measurability. We prove that, whenever $G$ is nondecreasing, every $G$-integrable function is $\mu_G$-measurable, where $\mu_G$ is the natural Radon measure induced by $G$. We also show that, for an arbitrary $G$ of bounded variation, every bounded $G$-integrable function is $\mu_G$-measurable.
%, where $|\mu_G|$ denotes the total variation measure of $\mu_G$. 
As an application, we provide a full characterization of Lebesgue integrablility with respect to Radon measures in terms of the $G$-integral, and demonstrate that the $G$-integral represents an extension of the Lebesgue integral with respect to $\mu_G$ for suitable $G$.

In addition, we establish a version of Hake's theorem for the $G$-integral in this setting. 
\end{abstract}

\section{Introduction}
The Henstock--Kurzweil integral \cite{Henstock1968}, \cite{Kurzweil1957} is an important analytic tool, as it extends the concept of integration on compact real intervals preserving the classical intuition underlying the Riemann integral based on Riemann sums, while at the same time it presents several remarkable new properties. In particular, the Henstock--Kurzweil integral generalizes the Lebesgue integral with respect to Lebesgue measure. However, unlike Lebesgue integrability, Henstock--Kurzweil integrability is closed by improper integration and allows for a more general Fundamental Theorem of Calculus. In addition, several convergence theorems such as the Monotone Covergence and the Dominated Convergence Theorems still hold for the Henstock--Kurzweil integral. On the other hand, a rich theory and vast set of applications emerge from the fact that, unlike the Lebesgue integral, the Henstock--Kurzweil integral is \emph{nonabsolute}, meaning that the integrability of a function $f$ does not guarantee the integrability of $|f|$.  %allows the integration of functions that lie beyond the scope of the Riemann and even Lebesgue theories; 
See, for example, \cite{Bartle}, \cite{KurtzSwartz}, \cite{MST} and \cite{Pfeffer}.

Over the past decades, several extensions and generalizations of the Henstock--Kurzweil integral have been developed; see, for example, \cite{BianconiKaufmann}, \cite{Lee}, and, notably, the development of $\Delta$- and $\nabla$-integrals on time scales in \cite{BohnerPeterson} and \cite{PetersonThompson}. The linear order of the real line plays a central role in the development of Henstock--Kurzweil integration on real compact intervals and its main time scale variants, as their nonabsoluteness is connected to conditionality. With this in account, in \cite{CanKau} we introduced a notion of Kurzweil--Stieltjes integration for functions defined on compact lines, extending the ideas of Henstock and Kurzweil to a broader class of ordered spaces.

In this paper, we continue the study initiated in \cite{CanKau} of the Kurzweil--Stieltjes integral, hereafter called the $G$-integral, on compact linearly ordered topological spaces, hereafter called compact lines. This class includes the usual compact intervals of the real line, but also nonmetrizable ordered spaces, such as compact ordinal spaces and richer examples such as the double arrow space; see \cite{Engelking}.

Integration on compact lines shares many similarities with the usual process of Henstock--Kurzweil integration on the real line. However, important differences arise as well. In particular, points may carry mass and may be isolated, which requires several classical arguments to be adapted to this setting. Moreover, compact lines need not be separable, first contable, or metrizable in general. This requires a new set of techniques for the development of a satisfactory theory and, on the other hand, reveals how much of the theory of Henstock--Kurzweil integration relies, actually, solely on order and completeness.

Our aim is to strengthen the connection between $G$-integration on compact lines and measure theory, more specifically the theory of Radon measures. The main point of contact between this topics relies on the fact that every right-continuous function $G$ of bounded variation on a compact line $K$ naturally induces a unique Radon measure $\mu_G$ on $K$; see \cite[Theorem 2.4.11]{Ronchim}, and also \cite{RonchimTausk} for applications to extensions of operators on Banach spaces of continuous functions). The collection of such functions forms a Banach space when endowed with a suitable norm, and will be denoted by $\mathrm{NBV}(K)$. The identification of $\mathrm{NBV}(K)$ with the Banach space $\mathcal{M}(K)$ of Radon measures on $K$ via the isometric isomorphism $G\mapsto \mu_G$ will be highlighted in the next section and is a crucial ingredient in the present work. Hereafter, by $\mu_G$ we always mean the Radon measure induced by $G\in \mathrm{NBV}(K)$. 

In \cite{CanKau}, a partial characterization was obtained relating Kurzweil--Stieltjes integration with respect to a function $G\in \mathrm{NBV}(K)$ on compact lines to (Lebesgue) integration with respect to $\mu_G$. This characterization is partial because measurability of the integrand and monotonicity of the integrator were assumed; see \cite[Theorem 5.6]{CanKau}. Our first result removes the measurability hypothesis in the monotone case.

\begin{theorem}\label{Thm:Main1}
Let $K$ be a compact line and let $G:K\to \mathbb{R}$ be a nondecreasing right-continuous function. If $f:K\to \mathbb{R}$ is $G$-integrable, then $f$ is $\mu_G$-measurable.
\end{theorem}
In particular, combining with results from \cite{CanKau}, we establish the $G$-integral as a proper extension of the Lebesgue integral with respect to positive Radon measures, see Corollary \ref{Cor:CharacterizationIncreasingMeasurability}. We also obtain a partial extension to the general case in which $G$ is not necessarily monotone.

\begin{theorem}\label{Thm:Main2}
Let $K$ be a compact line and let $G\in \mathrm{NBV}(K)$. %Let $\mu_G$ be the unique Radon measure on $K$ induced by $G$. 
If $f$ is $G$-integrable and bounded, then $f$ is $\mu_G$-measurable. Moreover, if $f$ and $|f|$ are $G$-integrable, then $f$ is $\mu_G$-measurable.
\end{theorem}
From Theorem \ref{Thm:Main2}, we also derive a characterization of Lebesgue integration with respect to  arbitrary Radon measures in terms of the $G$-integral, see Corollary \ref{Cor:CharacterizationMeasurability}.

To conclude, we present a version of Hake's theorem in this setting, demonstrating that the $G$-integral is closed by improper integration for a wide class of integrators.  

\begin{theorem}[Hake theorem]\label{Thm:Hake}
Let $K$ be a compact line, let $G:K\to\mathbb{R}$ be a right-continuous nondecreasing function, and let $f:K\to\mathbb{R}$ be a function. Assume that $1_K$ is left-dense and that there exists $A\in\mathbb{R}$ such that, for every $y\in(0_K,1_K)$, the function $f$ is $G$-integrable on $[0_K,y]$ and
\[
A=\lim_{y \nearrow 1_K} \int_K f_{[0_K,y]} \, dG
\quad
\left(=\lim_{y \nearrow 1_K} \int_{0_K}^y f \, dG\right).
\]
Then $f$ is $G$-integrable on $K$ and
\[
\int_K f \, dG
=
A+f(1_K)(G(1_K)-L_G(1_K)).
\]

Similarly, assume that $0_K$ is right-dense and that there exists $B\in\mathbb{R}$ such that, for every $y\in(0_K,1_K)$, the function $f$ is $G$-integrable on $(y,1_K]$ and
\[
B=\lim_{y \searrow 0_K} \int_K f_{(y,1_K]} \, dG.
\]
Then $f$ is $G$-integrable on $K$ and
\[
\int_K f \, dG
=
B-f(0_K)G(0_K).
\]
\end{theorem}

The paper is organized as follows. In Section~\ref{Sec:Prelim}, we briefly recall the concept of $G$-integration. In Section~\ref{Sec:Diff}, we revisit the notion of $G$-differentiability and the Fundamental Theorem of Calculus, which forms the core of the arguments developed later in Section~\ref{Sec:Measurable}. Section~\ref{Sec:Measurable} is devoted to measurability results for nondecreasing functions $G$ and to the proof of Theorem~\ref{Thm:Main1}. In Section~\ref{Sec:Hake}, we establish Theorem~\ref{Thm:Hake}. Finally, in Section~\ref{Sec:MeasurableNBV}, we use results from the previous sections to investigate consequences for general integrators in $\mathrm{NBV}(K)$, obtaining Theorem~\ref{Thm:Main2} and further applications.

\section{Basic terminology and preliminaries}
\label{Sec:Prelim}
The purpose of this section is to introduce the main objects used throughout the paper. In general, we follow the terminology established in \cite{CanKau}. By a \emph{compact line} we mean a compact linearly ordered topological space. If $K$ is a compact line, we denote its minimum and maximum elements by $0_K$ and $1_K$, respectively. Throughout this paper, we always assume that $0_K<1_K$ in order to avoid trivial cases. Given a point $x\in K$, if $x$ is left-isolated, we denote its immediate predecessor by $x^-$. If $x$ is right-isolated, we denote its immediate successor by $x^+$. A point $x\in(0_K,1_K]$ that is not left-isolated is said to be \emph{left-dense}. Similarly, a point $x\in[0_K,1_K)$ that is not right-isolated is said to be \emph{right-dense}; see \cite[Definition~2.1]{CanKau}.

We use the standard interval notation: parentheses denote open endpoints, and brackets denote closed endpoints. Unless explicitly stated otherwise, when we refer to open or closed intervals, we always mean them in the topological sense.

Given a compact line $K$, a \emph{tagged partition} (which we will simply call a partition) is any collection $P=\{([x_{i-1},x_i],t_i):i=1,\ldots,n\}$ such that
\begin{equation}\label{Rel:11111111}
0_K=x_0\leq x_1\leq \cdots \leq x_n=1_K,
\end{equation}
and $t_i\in [x_{i-1},x_i]$ for each $i$. Any collection of points $D=\{x_0,\ldots,x_n\}$ satisfying \eqref{Rel:11111111} will be called a division of $K$.

A \emph{gauge} on $K$ is a function $\delta$ that assigns to each $x\in K$ an open interval $\delta(x)$ containing $x$. We say that the partition $P$ is \emph{$\delta$-fine} if for every $1\leq i\leq n$ one has $[x_{i-1},x_i]\subset \delta(t_i)$.

Given functions $f$ and $G$ on $K$ and a partition $P=\{([x_{i-1},x_i],t_i):i=1,\ldots,n\}$, the \emph{Riemann sum} of $f$ with respect to $G$ and $P$ is defined by
\[
S(f,G,P)=f(0_K)G(0_K)+\sum_{i=1}^n f(t_i)\bigl(G(x_i)-G(x_{i-1})\bigr).
\]

We say that $f$ is \emph{Kurzweil--Stieltjes integrable} with respect to $G$, or simply $G$-integrable, if there exists $A\in\mathbb{R}$ such that for every $\varepsilon>0$ there exists a gauge $\delta$ on $K$ with the property that, for every $\delta$-fine partition $P$, one has
\[
|S(f,G,P)-A|<\varepsilon.
\]
In this case, we write
\[
A=\int_K f\,dG.
\]

For simplicity of notation, for each $\varepsilon>0$ we will say that a gauge $\delta$ is an $\varepsilon$-gauge for $f$ if $\left|S(f,G,P)-\int_K f\,dG\right|<\varepsilon$ whenever $P$ is a $\delta$-fine partition of $K$.

In this paper, we restrict our attention to functions $G$ for which one can obtain good integration properties. The first class of functions we consider are the so-called \emph{amenable} functions, that is, functions which are regulated and right-continuous. Recall that a function $G:K\to\mathbb{R}$ is \emph{regulated} if, for every left-dense point $x\in K$, the limit $\lim_{y\nearrow x}G(y)$ exists, and if $x$ is right-dense, then $\lim_{y\searrow x}G(y)=G(x)$.

We now recall an important function introduced in \cite{CanKau}, which plays a central role in $G$-integration. If $G$ a regulated function, we define $L_G:K\to \mathbb{R}$ by
\[
L_G(x)=
\begin{cases}
0, & \text{if } x=0_K,\\[0.5ex]
\lim_{y\nearrow x} G(y), & \text{if } x\in(0_K,1_K] \text{ is left-dense},\\[0.5ex]
G(x^-), & \text{if } x\in(0_K,1_K] \text{ is left-isolated}.
\end{cases}
\]

If $G$ is amenable and has bounded variation, we say that $G$ belongs to $\mathrm{NBV}(K)$, following the terminology established in \cite{Ronchim}. We now recall the notion of variation.

\begin{defin}\label{Def:Variation}
Let $K$ be a compact line, and let $D=\{x_0,x_1,\ldots,x_n\}$ be a division of $K$. We define
\[
\operatorname{Var}(G,D)=\sum_{i=1}^n |G(x_i)-G(x_{i-1})|,
\]
and the \emph{total variation} of $G$ by $\operatorname{Var}(G)=\sup_D \operatorname{Var}(G,D)$, where the supremum is taken over all divisions $D$ of $K$. We say that $G$ is of \emph{bounded variation} if $\operatorname{Var}(G)<\infty$.
\end{defin}

The space $\mathrm{NBV}(K)$ is a Banach space when endowed with the norm
\[
\|G\|=|G(0_K)|+\operatorname{Var}(G).
\]
Its importance is highlighted by the fact, proved in \cite[Theorem~2.4.11]{Ronchim}, that $\mathrm{NBV}(K)$ is linearly and isometrically isomorphic to the space of Radon measures on $K$, denoted by $\mathcal{M}(K)$, endowed with the total variation norm.

If $G\in \mathrm{NBV}(K)$, we shall always denote by $\mu_G$ the associated Radon measure, and by $|\mu_G|$ the total variation measure of $\mu_G$. In particular, for each $x\in K$, we have
\[
\mu_G([0_K,x])=G(x), \qquad \text{and} \qquad \mu_G([0_K,x))=L_G(x).
\]

We say that a subset $A\subset K$ is \emph{$G$-null} if its outer measure with respect to $|\mu_G|$ satisfies $|\mu_G|^*(A)=0$, see \cite[Definition 3.21]{CanKau}.

\section{Fundamental Theorem of Calculus Revisited}
\label{Sec:Diff}
In this section, we revisit a differentiation result established in \cite[Theorem 4.9]{CanKau}, which is fundamental for the developments of the present paper. We observed a gap in the original proof, specifically in the argument for Claim~1 concerning the use of the Vitali-type covering theorem \cite[Theorem 3.26]{CanKau}. For this reason, we provide below a revised and self-contained proof.

For the convenience of the reader, we first recall a definition introduced in \cite[Definition 4.4]{CanKau}, inspired in part by \cite[Definition 1.1]{PousoRodriguez}.

\begin{defin}\label{Def:Diff}
Let $K$ be a compact line, and let $f,G:K\to\mathbb{R}$ be functions, where $G$ is nondecreasing and right-continuous. Fix $x\in K$.

If $L_G(x)=G(x)$, we say that $f$ is \emph{$G$-differentiable at $x$} if the following conditions hold:
\begin{enumerate}[label=\textnormal{(D\alph*)}]
    \item If $x$ is left-isolated in $K$, then $f(x)=L_f(x)$;
    
    \item\label{it:diff_defin1} For every neighborhood $V$ of $x$, the function $G$ is not constant on $V$;

    \item\label{it:diff_defin2} There exists $D\in\mathbb{R}$ such that, for every $\varepsilon>0$, there exists a neighborhood $V$ of $x$ in $K$ satisfying
    \[
    \left|f(y)-f(x)-D\bigl(G(y)-G(x)\bigr)\right|
    \leq
    \varepsilon |G(y)-G(x)|
    \]
    for all $y\in V$.
\end{enumerate}

If $L_G(x)\neq G(x)$, we say that $f$ is \emph{$G$-differentiable at $x$} if:
\begin{enumerate}[label=\textnormal{(E\alph*)}]
    \item $f$ is regulated at $x$;

    \item There exists $D\in\mathbb{R}$ such that, for every $\varepsilon>0$, there exists a neighborhood $V$ of $x$ in $K$ satisfying
    \[
    \left|f(y)-L_f(x)-D\bigl(G(y)-L_G(x)\bigr)\right|
    \leq
    \varepsilon |G(y)-L_G(x)|
    \]
    for all $y\in V\cap[x,1_K]$.
\end{enumerate}

In both cases, the number $D$ is called the \emph{$G$-derivative} of $f$ at $x$, and we write
\[
\frac{df}{dG}(x)=D.
\]
\end{defin}

The gap in the proof of \cite[Theorem 4.9]{CanKau} arises because, in \cite[Definition 3.25]{CanKau}, the sets in an admissible covering are not required to have positive measure, a condition needed in \cite[Theorem 3.26]{CanKau}. This requirement is essential: if intervals of measure zero were allowed, the conclusion of \cite[Theorem 3.26]{CanKau} would fail in general, since admissible coverings could consist entirely of null sets. Therefore, we present below a corrected definition of admissible covering, excluding intervals of zero measure.

\begin{defin}
Let $K$ be a compact line, let $\mu$ be a positive Radon measure on $K$, and let $A\subset K$. An \emph{admissible covering} of $A$ is a family $\mathcal{F}$ of intervals of $K$ such that $\mu(J)>0$ for every $J\in\mathcal{F}$ and, whenever $a\in A$ and $\{J_1,\dots,J_k\}$ is a finite family of pairwise disjoint elements of $\mathcal{F}$ with $a\in A\setminus\bigcup_{i=1}^k J_i$,
there exists $I\in\mathcal{F}$ such that $a\in I$ and $I\cap\Bigl(\bigcup_{i=1}^k J_i\Bigr)=\emptyset$.
\end{defin}

\begin{theorem}[Fundamental Theorem of Calculus: differentiating integrals]\label{Thm:differentiation}
Let $K$ be a compact line, and let $G: K \to \mathbb{R}$ be a nondecreasing, positive, and right-continuous. Suppose that $f:K\to\mathbb{R}$ is $G$-integrable, and consider the function $F:K\to \mathbb{R}$ given by the formula
\[F(x) = \int_{0_K}^x f \, dG.\]
 Then there exists a $G$-null set $\mathcal{Z} \subset K$ such that, for all $x \in K \setminus \mathcal{Z}$, the derivative $\frac{dF}{dG}(x)$ exists and satisfies  
\[\frac{dF}{dG}(x) = f(x).\]
\end{theorem}
\begin{proof}
From \cite[Theorem 4.1]{CanKau}, $F$ as defined above is an amenable function. It follows from \cite[Proposition 4.5]{CanKau}, if $x \in K$ is such that $G(x) \neq L_G(x)$, the $F$ is $G$-differentiable at $x$, and, recalling again  \cite[Theorem 4.1]{CanKau}, we have
\[
\frac{dF}{dG}(x) = \frac{F(x) - L_F(x)}{G(x) - L_G(x)} = f(x).
\]

We deduce that the set $\mathcal{Z}$ of all points $x \in K$ where either $\frac{dF}{dG}(x)$ does not exist, or it exists but differs from $f(x)$, must satisfy $G(x) = L_G(x)$. Set $U = \{x \in K : G(x) = L_G(x)\}$.

From \ref{it:diff_defin1} and \ref{it:diff_defin2}, we can partition $\mathcal{Z}$ into two subsets $\mathcal{A}$ and $\mathcal{B}$ where: 
\begin{itemize}
    \item $\mathcal{A}$ consists of all points $x \in U$ for which there exists a neighborhood $V$ of $x$ such that $G$ is constant on $V$, and
    \item $\mathcal{B}$ is composed of all points $x \in U \setminus \mathcal{A}$ for which there exists $\alpha(x) > 0$ such that, for every neighborhood $V$ of $x$, there exists $c_{x,V} \in V \setminus \{x\}$ satisfying
\begin{equation}\label{RelAuxNotDif}
\left|F(x) - F(c_{x,V}) - f(x)\big(G(x) - G(c_{x,V})\big)\right| 
> \alpha(x) \left|G(x) - G(c_{x,V})\right|.
\end{equation}
\end{itemize}

To prove the theorem, it suffices to show that both sets $\mathcal{A}$ and $\mathcal{B}$ are $G$-null.

\medskip
\textbf{Claim 1:} $\mathcal{A}$ is $G$-null.

\medskip
For each $x \in \mathcal{A}$, choose an open interval $V_x$ containing $x$ such that $G$ is constant on $V_x$. Since $x \in U$, we have $G(x)=L_G(x)$. We claim that there exists an open interval $I_x$, containing $x$ and satisfying $I_x \subseteq V_x$, such that $\mu_G(I_x)=0$. Indeed, we distinguish four cases.

If $x$ is isolated in $K$, we may take $I_x=\{x\}$, which is open in $K$. Then
\[
\mu_G(I_x)=\mu_G(\{x\})=G(x)-L_G(x)=0.
\]

Assume next that $x$ is either $0_K$ or left-isolated, but right-dense. Since $V_x$ is open and contains $x$, there exists $b>x$ such that $[x,b]\subseteq V_x$. Set $I_x=[x,b)$. Since $G$ is constant on $V_x$ and moreover $L_G(x)=G(x)$, we have
\[
\mu_G(I_x)=L_G(b)-L_G(x)=G(b)-G(x)=0.
\]

Assume now that $x$ is either $1_K$ or right-isolated, but left-dense. Then there exists $a<x$ such that $[a,x]\subseteq V_x$. Set $I_x=(a,x]$. Since $G$ is constant on $V_x$, we have $G(a)=G(x)$ and therefore
\[
\mu_G(I_x)=G(x)-G(a)=0.
\]

Finally, assume that $x$ is both left-dense and right-dense. Then there exist $a<x<b$ such that $[a,b]\subseteq V_x$. Set $I_x=(a,b)$. Since $G$ is constant on $V_x$ we have
\[
\mu_G(I_x)=L_G(b)-G(a)=G(b)-G(a)=0.
\]

Our claim is established.

Set $\Omega=\bigcup_{x\in \mathcal{A}}I_x$ and let $W\subseteq \Omega$ be an arbitrary compact subset. Then $\{I_x:x\in \mathcal{A}\}$ is an open cover of $W$, so by compactness there exist $x_1,\dots,x_n\in \mathcal{A}$ such that $W\subseteq I_{x_1}\cup\cdots\cup I_{x_n}$. Hence
\[
\mu_G(W)\le \sum_{i=1}^n \mu_G(I_{x_i})=0.
\]
Therefore every compact subset of $\Omega$ has $\mu_G$-measure zero. Since $\mu_G$ is a Radon measure and $\Omega$ is open, it follows that $\mu_G(\Omega)=0$. Since $\mathcal{A}\subset \Omega$, we obtain that $\mu_G(\mathcal{A})=0$.

\medskip
\textbf{Claim 2:} $\mathcal{B}$ is $G$-null.
\medskip

Let $C$ be the subset consisting of all points $x \in \mathcal{B}$ such that, for every open interval $V$ containing $x$, there exists $c_{x,V} > x$ satisfying relation~\eqref{RelAuxNotDif}. In this case, we fix $y_{x,V} = x$ and $z_{x,V} = c_{x,V}$, and define the interval $J_{x,V} = [y_{x,V}, z_{x,V}]$.

 If $x \in \mathcal B \setminus C$, then there exists an open interval $V_0$ containing $x$ such that no point $c > x$ in $V_0$ satisfies \eqref{RelAuxNotDif}. Since $x \in \mathcal B$, for every open interval $V \subseteq V_0$ containing $x$, there must exist a point $c_{x,V} < x$ satisfying \eqref{RelAuxNotDif}. In this case, we set $y_{x,V} = c_{x,V}$ and $z_{x,V} = x$. If $G(y_{x,V}) = L_G(y_{x,V})$, define $J_{x,V} = [y_{x,V}, z_{x,V}]$;
otherwise, define $J_{x,V} = (y_{x,V}, z_{x,V}]$. 

For all $x$ and all open intervals $V$ containing $x$, we have
\begin{equation}\label{Rel:MeasureFinal}
\mu_G(J_{x,V}) = G(z_{x,V}) - G(y_{x,V}) > 0.
\end{equation}

Next, for each $n\in\mathbb N$, let $B_n=\{x\in B:\alpha(x)\geq 1/n\}$. Since $B=\bigcup_{n=1}^\infty B_n$, to conclude the proof it suffices to show that each $B_n$ is $G$-null.
    
Fix $n \in \mathbb{N}$, and let $\varepsilon > 0$ be arbitrary. Since $f$ is $G$-integrable, we may fix a $\frac{\varepsilon}{n}$-gauge $\delta$ for $f$ and consider the collection  
\[
\mathcal{F} = \{J_{x,V} : x \in B_n \text{ and } V\subseteq \delta(x)\text{ is an open interval containing } x\}.
\]
We claim that this family forms an admissible covering of $B_n$. Indeed, let $J_{x_1,V_1} < J_{x_2,V_2} < \dots < J_{x_m,V_m}$
be pairwise disjoint elements of $\mathcal{F}$, and let $a \in B_n \setminus \bigcup_{i=1}^m J_{x_i,V_i}$. We claim that there exists an open interval $V$ containing $a$ such that $V \cap \left(\bigcup_{i=1}^m J_{x_i,V_i}\right)=\emptyset$. Otherwise, every open interval containing $a$ would intersect $\bigcup_{i=1}^m J_{x_i,V_i}$. Since the intervals $J_{x_i,V_i}$ are pairwise disjoint and linearly ordered, and since $a\notin \bigcup_{i=1}^m J_{x_i,V_i}$,
it would follow that there exists $1\le i\le m$ such that
\[J_{x_i,V_i}=(y_{x_i,V_i},z_{x_i,V_i}]
\qquad\text{and}\qquad
a=y_{x_i,V_i}.
\]
By the definition of $J_{x_i,V_i}$, this implies $L_G(a)<G(a)$, which contradicts the fact that $a\in B_n\subseteq \mathcal B\subseteq U$. Thus there exists an open interval $V$ containing $a$ such that $V \cap \left(\bigcup_{i=1}^m J_{x_i,V_i}\right)=\emptyset$. Setting $V'=V\cap \delta(a)$, we have $J_{a,V'}\in\mathcal F$, and clearly
\[
J_{a,V'}\cap \left(\bigcup_{i=1}^m J_{x_i,V_i}\right)=\emptyset.
\]
This proves that $\mathcal F$ is an admissible covering of $B_n$.

Next, according to \cite[Theorem 3.26]{CanKau}, there exists a finite subcollection $\mathcal{H} = \{J_{x_1,V_1}, \dots, J_{x_m,V_m}\} \subset \mathcal{F}$ consisiting of pairwise disjoint intervals such that $\mu_G^*(B_n \setminus \bigcup \mathcal{H}) < \varepsilon$,
where $\mu_G^*$ denotes the exterior measure induced by $\mu_G$.
On one hand, by employing \cite[Theorem 3.16]{CanKau}, we obtain
\begin{align*}
    &\sum_{i=1}^m\left|f(x_i)(G(z_{x_i,V_i}) - G(y_{x_i,V_i})) + f(y_{x_i,V_i})G(y_{x_i,V_i}) - \int_{y_{x_i,V_i}}^{z_{x_i,V_i}} f\,dG\right|\\
    &= \sum_{i=1}^m\left|f(x_i)(G(z_{x_i,V_i}) - G(y_{x_i,V_i})) - (F(z_{x_i,V_i}) - F(y_{x_i,V_i}))\right| \\
    &> \sum_{i=1}^m \alpha(x_i)\big(G(z_{x_i,V_i}) - G(y_{x_i,V_i})\big) \\
    &\geq \frac{1}{n} \sum_{i=1}^m \big(G(z_{x_i,V_i}) - G(y_{x_i,V_i})\big).
\end{align*}

On the other hand, observe that $\mathcal{S} = \{([y_{x_1,V_1}, z_{x_1,V_1}], x_1), \dots, ([y_{x_m,V_m}, z_{x_m,V_m}], x_m)\}$ is a $\delta$-fine tagged system in $K$. It then follows from  Saks--Henstock Lemma \cite[Corollary 3.19]{CanKau} that
\[\sum_{i=1}^m \left|f(x_i)(G(z_{x_i,V_i}) - G(y_{x_i,V_i})) + f(y_{x_i,V_i})G(y_{x_i,V_i}) - \int_{y_{x_i,V_i}}^{z_{x_i,V_i}} f\,dG\right| \leq \frac{4\varepsilon}{n}.\]

Combining these inequalities and recalling relation~\eqref{Rel:MeasureFinal}, we get
\[\mu_G\left(\bigcup \mathcal{H}\right) = \sum_{i=1}^m \big(G(z_{x_i,V_i}) - G(y_{x_i,V_i})\big) < 4\varepsilon.\]
Therefore,
\[\mu_G^*(B_n) \leq \mu_G^*\left(B_n \setminus \bigcup \mathcal{H}\right) + \mu_G^*\left(\bigcup \mathcal{H}\right) < 5\varepsilon,\]
which concludes the proof.
\end{proof}

\section{Measurability of $G$-integrable functions}
\label{Sec:Measurable}

In this section, our goal is to establish Theorem~\ref{Thm:Main1}. A natural approach would be to adapt the classical argument for the usual Henstock--Kurzweil integral by constructing a sequence of step functions converging pointwise to a $G$-integrable function $f$. However, this method fails in the present setting, since a nondecreasing right-continuous function $G\colon K\to\mathbb{R}$ may have many plateaus, that is, intervals on which $G$ is constant. On such regions, the $G$-derivative may assume different values at the endpoints. For this reason, a more flexible strategy is required, inspired by \cite[Corollary 4.8.5]{KurtzSwartz}; see also \cite[Theorem 6.3.3]{Pfeffer}. This strategy will be developed through a sequence of auxiliary lemmas.
We begin with the following definition.

\begin{defin}\label{Def:Step}
Let $K$ be a compact line. A function $S:K\to\mathbb{R}$ is called a \emph{step function} if there exist a division of $K$ $0_K=s_0<s_1<\cdots<s_n=1_K$ such that $S$ is constant on each interval $(s_{k-1},s_k)$ for $k=1,\dots,n$.
\end{defin}

The following lemma extends a classical approximation result to compact lines. For the standard version on real intervals, see \cite[Theorem 3.17]{Bartle}.

\begin{lemma}\label{Lem:RegulatedIsMeasurable}
Let $K$ be a compact line and let $G:K\to\mathbb{R}$ be a regulated function. Then $G$ is the uniform limit of step functions. In particular, $G$ is Borel measurable.
\end{lemma}
\begin{proof}
Fix $\varepsilon>0$. For each $w\in K$, since $G$ is regulated at $w$, we can choose a closed interval $J_w=[a_w,b_w]\subset K$
such that $w\in \operatorname{int}_K(J_w)$ and:
\begin{enumerate}[label=\textnormal{(\roman*)}]
    \item if $w\neq 0_K$, then
    \[
    y_1,y_2\in J_w\cap [0_K,w)
    \quad\Longrightarrow\quad
    |G(y_1)-G(y_2)|<\frac{\varepsilon}{2};
    \]
    \item if $w\neq 1_K$, then
    \[
    y_1,y_2\in J_w\cap (w,1_K]
    \quad\Longrightarrow\quad
    |G(y_1)-G(y_2)|<\frac{\varepsilon}{2}.
    \]
\end{enumerate}
The collection $\{\operatorname{int}_K(J_w):w\in K\}$ is an open cover of $K$. By compactness, there exist $w_1,\dots,w_m\in K$
such that $K=\bigcup_{j=1}^m \operatorname{int}_K(J_{w_j})$. 

We let $F=\bigcup_{j=1}^m\{a_{w_j},w_j,b_{w_j}\}$. By enumerating the distinct points of $F$ in increasing order, we obtain a a division of $K$:
\[
0_K=s_0<s_1<\cdots<s_n=1_K.
\]

We claim that for each $k=1,\ldots,n$, there exists some index $j=1,\ldots,m$ such that either 
\[
(s_{k-1},s_k)\subset J_{w_j} \cap [0_K,w_j)
\qquad\text{or}\qquad
(s_{k-1},s_k)\subset J_{w_j} \cap (w_j,1_K].
\]

Indeed, let $x\in (s_{k-1},s_k)$ be arbitrary. Since the interiors $\operatorname{int}_K(J_{w_j})$ cover $K$, there exists $j$ such that $x\in \operatorname{int}_K(J_{w_j})$. Because the points $a_{w_j},w_j,b_{w_j}$ all belong to $F$, none of them belongs to $(s_{k-1},s_k)$, and we deduce that either 
\[
(s_{k-1},s_k)\subset (a_{w_j},w_j),\qquad\text{or}\qquad
(s_{k-1},s_k)\subset (w_j,b_{w_j}),
\]
which establishes our claim.

For each $k=1,\ldots, n$, if $(s_{k-1},s_k)\neq \emptyset$, we fix $x_k\in (s_{k-1},s_k)$. We define the function $S_\varepsilon:K\to\mathbb{R}$ as follows: 
\[
S_\varepsilon(x)=
\begin{cases}
G(s_k), & \text{if }x=s_k\text{ for some }k=0,\ldots,n,\\
G(x_k), & \text{if }x\in (s_{k-1},s_k)\text{ for some }k=1,\ldots,n.
\end{cases}
\]

It is clear that $S_\varepsilon$ is a well-defined step function.

Let $x\in K$ be arbitrary. If $x=s_k$ for some $k=0,\ldots,n$, then 
\[
|G(x)-S_\varepsilon(x)|=|G(s_k)-G(s_k)|=0.
\]
Otherwise, $x\in (s_{k-1},s_k)$ for some $k=1,\ldots,n$. As shown above, there is some $j=1,\ldots,m$ such that either $(s_{k-1},s_k)\subset J_{w_j}\cap [0_K,w_j)$ or $(s_{k-1},s_k)\subset J_{w_j}\cap (w_j,1_K]$. Therefore, by construction of the set $J_{w_j}$ we have
\[
|G(x)-S_\varepsilon(x)|=|G(x)-G(x_k)|<\frac{\varepsilon}{2}.
\]
Therefore, $\|G-S_\varepsilon\|_\infty<\varepsilon$. Since $\varepsilon>0$ is arbitrary, we deduce that $G$ is the uniform limit of step functions. In particular, $G$ is Borel measurable.
\end{proof}

In the following definition, we introduce functions that will play a fundamental role in the flexible strategy used in the proof of Theorem~\ref{Thm:Main1}.

\begin{lemma}\label{Lem:Functions}
Let $K$ compact line, and let $G:K\to\mathbb R$ be a nondecreasing function. For each $n\in\mathbb N$, we let $u_n,v_n:K\to K$ be the functions defined by
\[
u_n(x)=
\begin{cases}
\sup\{y\in K: y<x \text{ and } G(x)-G(y)\ge 2^{-n}\}, & \text{if the defining set is nonempty},\\
0_K, & \text{otherwise},
\end{cases}
\]
and
\[
v_n(x)=
\begin{cases}
\inf\{y\in K: y>x \text{ and } G(y)-G(x)\ge 2^{-n}\}, & \text{if the defining set is nonempty},\\
1_K, & \text{otherwise}.
\end{cases}
\]
For each $n \in \mathbb{N}$, $u_n,v_n$ are nondecreasing functions. Hence Borel measurable functions.
\end{lemma}
\begin{proof}
We first check that each $u_n:K\to K$ is nondecreasing. Indeed, let $x_1<x_2$ be arbitrary. If $\{y\in K: y<x \text{ and } G(x_1)-G(y)\ge 2^{-n}\}=\emptyset$, then $u_n(x_1)=0_K\leq u_n(x_2)$ and we are done. Otherwise, there is $y<x_1$ such that
\[
G(x_1)-G(y)\ge 2^{-n}.
\]
Since $G$ is nondecreasing we have
\[
G(x_2)-G(y)\ge G(x_1)-G(y)\ge 2^{-n},
\]
and we deduce that 
\[ \{y\in K: y<x \text{ and } G(x_1)-G(y)\ge 2^{-n}\}\subset  \{y\in K: y<x \text{ and } G(x_2)-G(y)\ge 2^{-n}\}.\] By taking the supremum we obtain $u_n(x_1)\le u_n(x_2)$.

Similarly, we check that each $v_n:K\to K$ is nondecreasing. If $x_1<x_2$ and  $\{y\in K: y<x \text{ and } G(y)-G(x_2)\ge 2^{-n}\}=\emptyset$, we have $ v_n(x_1)\leq 1_K=v_n(x_2)$ and we are done. Otherwise, there is $y>x_2$ such that
\[
G(y)-G(x_2)\ge 2^{-n},
\]
and since $G$ is nondeacreasing,
\[
G(y)-G(x_1)\ge G(y)-G(x_2)\ge 2^{-n}.
\]
Therefore, 
\[ \{y\in K: y<x \text{ and } G(y)-G(x_2)\ge 2^{-n}\}\subset  \{y\in K: y<x \text{ and } G(y)-G(x_1)\ge 2^{-n}\}\]
and we deduce by taking the infimum that $v_n(x_1)\le v_n(x_2)$.

We are done since monotone functions on compact lines are measurable.
\end{proof}

\begin{lemma}\label{Lem:Intersection}
Let $K$ be a compact line, let $G:K\to\mathbb R$ be nondecreasing, and fix $x\in K$. Assume that $G^{-1}(\{G(x)\})=[a,b]$ with $a\leq b$. If $a$ is left-dense, then $u_n(x)\nearrow a$, and if $b$ is right-dense, then $v_n(x)\searrow b$. Additionally, if $G$ is continuous at $a$, then $u_n(x)<a$ for all $n\in\mathbb N$, and if $G$ is continuous at $b$, then $v_n(x)>b$ for all $n\in\mathbb N$.
\end{lemma}
\begin{proof}
For each $n\in\mathbb N$, the following hold:
\begin{align*}
\{y\in K: y<x \text{ and } G(x)-G(y)\ge 2^{-n}\}
&\subset \{y\in K: y<x \text{ and } G(x)-G(y)\ge 2^{-(n+1)}\},\\
\{y\in K: y>x \text{ and } G(y)-G(x)\ge 2^{-n}\}
&\subset \{y\in K: y>x \text{ and } G(y)-G(x)\ge 2^{-(n+1)}\}.
\end{align*}
Therefore $u_n(x)\le u_{n+1}(x)$ and $v_{n+1}(x)\le v_n(x)$, so $(u_n(x))_n$ is nondecreasing and $(v_n(x))_n$ is nonincreasing.

Assume that $a$ is left-dense. Let $y\in (0_K,a)$ be arbitrary. Since $G(y)<G(a)=G(x)$, there exists $n_0\in\mathbb N$ such that $G(y)+2^{-n_0}<G(x)$. Hence, for $n\ge n_0$, we have $y\le u_n(x)\le a$. Since $(u_n(x))_n$ is nondecreasing and $a$ is left-dense, we conclude that $u_n(x)\nearrow a$.

If, in addition, $G$ is continuous at $a$, then $u_n(x)<a$ for all $n\in\mathbb N$. Indeed, if $u_n(x)=a$ for some $n$, since
$u_n(x)=\sup\{y\in K: y<x \text{ and } G(x)-G(y)\ge 2^{-n}\}$, there exists a net $(y_\alpha)_\alpha$ such that $y_\alpha<x$, $G(x)-G(y_\alpha)\ge 2^{-n}$ for all $\alpha$, and $y_\alpha\nearrow a$. By continuity,
\[
G(y_\alpha)\to G(a)=G(x),
\]
which contradicts the fact that $G(y_\alpha)\le G(x)-2^{-n}$.

Now assume that $b$ is right-dense. Let $y\in (b,1_K)$ be arbitrary. Since $G(x)=G(b)<G(y)$, there exists $n_0\in\mathbb N$ such that $G(x)+2^{-n_0}<G(y)$. Hence, for $n\ge n_0$, we have $b\le v_n(x)\le y$. Since $(v_n(x))_n$ is nonincreasing and $b$ is right-dense, we conclude that $v_n(x)\searrow b$.

If, in addition, $G$ is continuous at $b$, then $v_n(x)>b$ for all $n\in\mathbb N$. The argument is analogous to the previous case.
\end{proof}

\begin{lemma}\label{Lem:Borel_fn}
Let $K$ be a compact line, and let $F,G:K\to\mathbb R$ be Borel measurable functions. Assume that $G$ is nondecreasing. For each $x\in K$, let
\[
U(x)=\sup_{n\in\mathbb N}u_n(x),\qquad V(x)=\inf_{n\in\mathbb N}v_n(x),
\]
and consider the sets 
\[
A=\{x\in K: U(x)=x<V(x)\},\qquad
B=\{x\in K: U(x)<x=V(x)\},
\]
and
\[
C=\{x\in K: U(x)=x=V(x)\}.
\]
For each $n \in \mathbb{N}$ define $f_n=\ell_n\chi_A+r_n\chi_B+q_n\chi_C$, where
\[
\ell_n(x)=
\begin{cases}
\displaystyle \frac{F(x)-F(u_n(x))}{G(x)-G(u_n(x))}, & \text{if } G(x)>G(u_n(x)),\\[1em]
0, & \text{otherwise},
\end{cases}
\]
\[
r_n(x)=
\begin{cases}
\displaystyle \frac{F(v_n(x))-F(x)}{G(v_n(x))-G(x)}, & \text{if } G(v_n(x))>G(x),\\[1em]
0, & \text{otherwise},
\end{cases}
\]
and
\[
q_n(x)=
\begin{cases}
\displaystyle \frac{F(v_n(x))-F(u_n(x))}{G(v_n(x))-G(u_n(x))}, & \text{if } G(v_n(x))>G(u_n(x)),\\[1em]
0, & \text{otherwise}.
\end{cases}
\]
Then $f_n$ is Borel measurable for every $n\in\mathbb N$.
\end{lemma}

\begin{proof}
From Lemma~\ref{Lem:Functions}, we know that, for every $n\in\mathbb N$, the functions $u_n$ and $v_n$ are nondecreasing, and hence Borel measurable. Moreover, by Lemma~\ref{Lem:Intersection}, for each $x\in K$ the sequence $(u_n(x))_{n\in\mathbb N}$ is nondecreasing, while $(v_n(x))_{n\in\mathbb N}$ is nonincreasing. Since every compact line is Dedekind complete; see \cite[Theorem 2.1.5]{Ronchim}; the functions $U,V:K\to K$ defined by
\[
U(x)=\sup_{n\in\mathbb N}u_n(x)=\lim_{n\to\infty}u_n(x), \qquad
V(x)=\inf_{n\in\mathbb N}v_n(x)=\lim_{n\to\infty}v_n(x)
\]
are well defined. Furthermore, as the pointwise supremum and infimum of nondecreasing functions, respectively, both $U$ and $V$ are nondecreasing. Therefore, $U$ and $V$ are Borel measurable.

Next, since $F$ and $G$ are Borel measurable by assumption, their compositions with $u_n$ and $v_n$ are also Borel measurable. In particular, the functions
\[
x\mapsto F(u_n(x)),\qquad x\mapsto F(v_n(x)),\qquad
x\mapsto G(u_n(x)),\qquad x\mapsto G(v_n(x))
\]
are Borel measurable. Consequently, the functions $\ell_n$, $r_n$, and $q_n$ are Borel measurable as well.

Now let
\[
D=\{x\in K:\, U(x)<x\},\qquad
E=\{x\in K:\, x<V(x)\}.
\]
Since $U$, $V$, and the identity map on $K$ are Borel measurable, the sets $D$ and $E$ are Borel. Moreover, as $U(x)\le x\le V(x)$ for every $x\in K$, we have
\[
\{x\in K:\, U(x)=x\}=K\setminus D,
\qquad
\{x\in K:\, V(x)=x\}=K\setminus E.
\]
Therefore,
\[
A=(K\setminus D)\cap E,\qquad
B=D\cap(K\setminus E),\qquad
C=(K\setminus D)\cap(K\setminus E).
\]
Hence the sets $A$, $B$, and $C$ are Borel.

Finally, since $\chi_A$, $\chi_B$, and $\chi_C$ are Borel measurable and $\ell_n$, $r_n$, and $q_n$ are Borel measurable, the function
\[
f_n=\ell_n\chi_A+r_n\chi_B+q_n\chi_C
\]
is Borel measurable. This completes the proof.
\end{proof}

We are now ready to prove Theorem~\ref{Thm:Main1}.

\begin{proof}[Proof of Theorem \ref{Thm:Main1}]
We assume that $G$ is not constant, since otherwise the conclusion is immediate. For each $n \in \mathbb{N}$, let 
$f_n:K\to \mathbb{R}$ be the measurable function given by Lemma~\ref{Lem:Borel_fn}, defined for the function $F:K\to \mathbb{R}$,
\[
F(x)=\int_{0_K}^{x}f\, dG.
\]
We know from \cite[Theorem 4.1]{CanKau} that $F$ is amenable, hence Borel measurable.

The next step is to isolate an appropriate measurable set $\mathcal{Z}$ such that, for each $x \in K \setminus \mathcal{Z}$,
\begin{equation}\label{Rel:Aux3}
f(x)=\lim_{n\to \infty}f_n(x).
\end{equation}

We proceed as follows. Since $G$ is nondecreasing, the set $Q=\{b\in K:\, G(b)>L_G(b)\}$ is countable. Let $W$ be the set of all points $x\in K$ such that $G^{-1}(\{G(x)\})=[x,b_x)$ for some $b_x>x$. This also includes the situation in which $G^{-1}(\{G(x)\})=[x,b]$,
with $b$ right-isolated as in that case we set $b_x=b^+$. For each $x\in W$, we have $G(b_x)>L_G(b_x)$, hence $b_x\in Q$. Moreover, the map $x\mapsto b_x$ is injective. Indeed, if $x<y$ and $b_x=b_y$, then $y\in[x,b_x)$, so $G(y)=G(x)$, which implies $x\in G^{-1}(\{G(y)\})=[y,b_y)$, a contradiction. Therefore, $W$ is countable.

We set $\mathcal A=Q\cup W\cup\{0_K,1_K\}$. We note that $\mathcal A$ need not be $G$-null. Nevertheless, it is $G$-measurable as it is countable.

Additionally, by Theorem~\ref{Thm:differentiation}, there exists a $G$-null set $\mathcal{B} \subset K$ such that, for every $x \in K \setminus \mathcal{B}$, $F$ is $G$-differentiable at $x$ with
\[
\frac{d F}{dG}(x) = f(x).
\]

We define
\[
\mathcal{Z} = \mathcal{A} \cup \mathcal{B}.
\]

Let $x\in K$ be arbitrary, and set $\lambda_x=G(x)$. Since $G$ is nondecreasing and right-continuous, the plateau $G^{-1}(\{\lambda_x\})$ must be of one of the following forms: an interval of the form $[a,b)$ with $a<b$; or an interval of the form $[a,b]$ with $a\leq b$ and $b$ not right-isolated. If the interval is of the form $[a,b]$ and $b$ is right-isolated, then $[a,b]=[a,b^+)$, so we are reduced to the previous case.

We assume that $x\in K\setminus\mathcal Z$. We first show that in this case $G^{-1}(\{\lambda_x\})$ cannot be of the form $[a,b)$. Indeed, suppose that $a<x<b$. Then $G$ is constant on a neighborhood of $x$, which implies that $F$ is not $G$-differentiable at $x$ according to \ref{it:diff_defin1} in Definition~\ref{Def:Diff}, and hence $x\in\mathcal B$, a contradiction. Therefore, we must have $a=x$. But in this case $G^{-1}(\{G(x)\})=[x,b)$, so $x\in W\subset \mathcal A$, again a contradiction. Hence, this case cannot occur.

Next, suppose that $G^{-1}(\{\lambda_x\})$ is of the form $[a,b]$, where $a\leq b$ and $b$ is not right-isolated. First, we note that we cannot have $a<x<b$, since, as in the previous case, this would imply that $x\in\mathcal B$, a contradiction. Therefore, either $x=a$ or $x=b$. We claim that, if $x$ is the left endpoint, then $x$ cannot be left-isolated.

Indeed, suppose that  $a=x$ and $x$ is left-isolated. If $L_G(x)<G(x)$, then $x\in Q\subset\mathcal A$, which is a contradiction. Hence $L_G(x)=G(x)$. But since $x$ is left-isolated, we also have $L_G(x)=G(x^-)=G(x)$, so that $x^-\in G^{-1}(\{\lambda_x\})$, contradicting the fact that $G^{-1}(\{\lambda_x\})=[x,b]$. 

So we are left with the following  cases to consider:

\medskip
\noindent
\textbf{Case 1:} $G^{-1}(\{\lambda_x\})=\{x\}$ and $x$ is neither left-isolated nor right-isolated.
\medskip

Since $0_K<x<1_K$ and $x\notin\mathcal A$, we have that $x$ is left-dense and right-dense, and that $L_G(x)=G(x)$. It follows, since $G$ is right-continuous, that $G$ is continuous at $x$. By Lemma~\ref{Lem:Intersection}, we have $u_n(x)\nearrow x$, $v_n(x)\searrow x$, and $u_n(x)<x<v_n(x)$ for every $n\in\mathbb N$. Consequently, $G(u_n(x))<G(x)<G(v_n(x))$ for every $n\in\mathbb N$.

Let $\varepsilon>0$ be arbitrary. Since $F$ is $G$-differentiable at $x$ with $\frac{dF}{dG}(x)=f(x)$, by the Straddle Lemma \cite[Lemma 4.6]{CanKau}, there exists an open interval $J$ containing $x$ such that, for any $(u,v]\subset J$ with $u \leq x \leq v$, we have
\[
\left| F(v) - F(u) - f(x)\big(G(v) - G(u)\big) \right|
\leq \varepsilon \big(G(v) - G(u)\big).
\]
Since $u_n(x)\nearrow x$ and $v_n(x)\searrow x$, there exists $n_0 \in \mathbb{N}$ such that $(u_n(x),v_n(x)]\subset J$ for all $n \geq n_0$. Therefore,
\[
|f_n(x)-f(x)|=|q_n(x)-f(x)|=\left| \frac{F(v_n(x)) - F(u_n(x))}{G(v_n(x)) - G(u_n(x))} - f(x) \right| \leq \varepsilon,
\]
for all $n \ge n_0$, which yields \eqref{Rel:Aux3}.

\medskip
\noindent
\textbf{Case 2:} $G^{-1}(\{\lambda_x\})=[a,b]$ with $a<b$ and $b$ is not right-isolated.
\medskip

As explained above, we either have $a=x$ or $b=x$. We first consider the case $G^{-1}(\{\lambda_x\})=[x,b]$ with $x<b$. Since $0_K<x$, we have that $x$ is left-dense, and since $x\notin\mathcal A$, we have $L_G(x)=G(x)$. It follows that $G$ is continuous at $x$, because $G$ is right-continuous. By Lemma~\ref{Lem:Intersection}, we have $u_n(x)\nearrow x$ and $u_n(x)<x$ for each $n\in\mathbb N$. Therefore,
\[
G(x)>G(u_n(x))
\qquad \text{ for each } n\in\mathbb N.
\]
It follows, since $F$ is $G$-differentiable at $x$ with $\frac{dF}{dG}(x)=f(x)$, that
\[
\lim_{n\to \infty}|f_n(x)-f(x)|=\lim_{n\to \infty}|\ell_n(x)-f(x)|=\lim_{n\to \infty}\left| \frac{F(x) - F(u_n(x))}{G(x) - G(u_n(x))} - f(x) \right|=0,
\]
which yields \eqref{Rel:Aux3}.

Next, consider the case $G^{-1}(\{\lambda_x\})=[a,x]$ with $a<x$. Since $x<1_K$, it follows that $x$ is right-dense, and in this case we also have that $G$ is continuous at $x$. By Lemma~\ref{Lem:Intersection}, we have $v_n(x)\searrow x$ with $v_n(x)>x$ for each $n\in\mathbb N$, which implies
\[
G(v_n(x))>G(x)
\qquad \text{for each } n\in\mathbb N.
\]
It follows, since $F$ is $G$-differentiable at $x$ with $\frac{dF}{dG}(x)=f(x)$, that
\[
\lim_{n\to \infty}|f_n(x)-f(x)|=\lim_{n\to \infty}|r_n(x)-f(x)|=\lim_{n\to \infty}\left| \frac{F(v_n(x)) - F(x)}{G(v_n(x)) - G(x)} - f(x) \right|=0,
\]
which yields \eqref{Rel:Aux3}.

To conclude the proof, define $\hat{f}:K \to \mathbb{R}$ by
\[
\hat{f}(x) = \limsup_{n \to \infty} f_n(x).
\]
Then $\hat{f}$ is measurable, being the pointwise $\limsup$ of measurable functions. Moreover, as established above, $\hat{f}(x)=f(x)$ for every $x\in K\setminus\mathcal Z$. Now define $g,h:K\to\mathbb R$ by
\[
g=f\,\chi_{\mathcal A}, \qquad h=\hat{f}\,\chi_{K\setminus\mathcal A}.
\]
Since $\mathcal{A}$ is countable, both $g$ and $h$ are measurable. Observe that $f(x) = g(x) + h(x)$ for all $x \in K \setminus \mathcal{B}$. Hence, the set $\{x \in K : f(x) \neq g(x) + h(x)\}$ is contained in $\mathcal{B}$, which is $G$-null. It follows that $f$ is $\mu_G$-measurable.
\end{proof}

As a result of Theorem~\ref{Thm:Main1}, one immediately obtains the following generalization of \cite[Theorem 5.6]{CanKau}.

\begin{corollary}\label{Cor:CharacterizationIncreasingMeasurability}
Let $K$ be a compact line, and let $G:K\to\mathbb R$ be a nondecreasing right-continuous  function. Let $\mu_G\in\mathcal M(K)$ be the Radon measure induced by $G$. Then, for every function $f:K\to\mathbb R$, the following are equivalent:
\begin{enumerate}
    \item $f\in L_1(\mu_G)$;
    \item both $f$ and $|f|$ are $G$-integrable.
\end{enumerate}
In this case,
\[
\int_K f\,dG=\int_K f\,d\mu_G.
\]
\end{corollary}
\begin{proof}
Without loss of generality, we may assume that $G$ is positive. Otherwise, we replace $G$ by $G+|G(0_K)|$. If $f$ is $\mu_G$-integrable, then $f$ is $\mu_G$-measurable, so the conclusion follows from \cite[Theorem 5.6]{CanKau}.

Conversely, if both $f$ and $|f|$ are $G$-integrable, then by Theorem~\ref{Thm:Main1} both $f$ and $|f|$ are $\mu_G$-measurable. Hence $f$ is measurable, and \cite[Theorem 5.6]{CanKau} applies again, yielding that $f$ is Lebesgue $\mu_G$-integrable and that
\[\int_K f\,dG=\int_K f\,d\mu_G.\]
\end{proof}

\section{Hake's Theorem for the $G$-Integral}
\label{Sec:Hake}

In this section, we establish a Hake-type theorem for the $G$-integral. In other words, $G$-integration on compact lines can be recovered from suitable one-sided limits of integrals of restricted functions.

We first recall an important fact concerning the $G$-integral in order to avoid possible confusion. Given an interval $I$ of a compact line $K$, we write
\[
f_I=f\chi_I,
\]
where $\chi_I$ denotes the indicator function of $I$.

By the results of \cite[\S 3.4]{CanKau}, if $I$ is compact, then $f_I$ is $G$-integrable if and only if the restriction $f|_I$ is $G|_I$-integrable (for simplicity, we shall say that $f|_I$ is $G$-integrable). However, the corresponding integrals need not coincide; see \cite[Proposition 3.15]{CanKau}.

In the first part of the proof of Theorem~\ref{Thm:Hake}, we shall use that
\[
\int_K f_{[0_K,y]}\,dG=\int_{0_K}^y f\,dG,
\]
whereas in the second part
\[
\int_K f_{(y,1_K]}\,dG=\int_K f_{[y,1_K]}\,dG-f(y)\int_K \chi_{\{y\}}\,dG=\int_K f_{[y,1_K]}\,dG-f(y)(G(y)-L_G(y)).
\]

The following elementary lemma will be useful in the proof of the next theorem. It shows that, at the endpoints of a compact line, either $G$ is locally constant or the point has a countable local basis.

\begin{lemma}\label{Lem:FirstCountableCompactLines}
Let $K$ be a compact line, and let $G:K\to\mathbb{R}$ be a nondecreasing function. If $a=0_K$ or $a=1_K$, then either
\begin{itemize}
\item there exists a neighborhood $V$ of $a$ such that $V\setminus \{a\}\neq \emptyset$, and $G$ is constant on $V\setminus \{a\}$, or
\item $a$ has a countable local basis.
\end{itemize}
\end{lemma}
\begin{proof}
If $0_K=1_K$, then $K$ consists of a single point and the conclusion is trivial. Thus we may assume that $0_K\neq 1_K$.

We assume that $a=1_K$ (the case $a=0_K$ is analogous). If $1_K$ is isolated, then it has a countable local basis, so we may assume that $1_K$ is not isolated.

Redefining $G(1_K)$ if necessary, we may assume that $L_G(1_K)=G(1_K)$. This does not affect the conclusion of the lemma, since changing the value of $G$ at $1_K$ does not alter whether $G$ is constant on $V\setminus\{1_K\}$ for some neighborhood $V$ of $1_K$.

Suppose that $G$ is not constant on any neighborhood of $1_K$. Then $G(1_K)>G(0_K)$, since otherwise $G$ would be constant on all of $K$. For each $n\in \mathbb{N}$, choose $a_n<1_K$ such that
\[
G(1_K)-\frac{G(1_K)-G(0_K)}{n}<G(a_n)<G(1_K).
\]

We claim that the family $\mathcal{V}=\{(a_n,1_K]:n\in \mathbb{N}\}$ is a local basis at $1_K$. Let $U$ be a neighborhood of $1_K$. Then there exists $c<1_K$ such that $(c,1_K]\subset U$. Since $G$ is nondecreasing and not constant on any neighborhood of $1_K$, there exists $x>c$ with $G(x)>G(c)$. As $G(a_n)\to G(1_K)$, we can choose $n$ such that $G(a_n)>G(c)$. Since $G$ is nondecreasing, it follows that $a_n>c$, and therefore $(a_n,1_K]\subset (c,1_K]\subset U$. Thus $\mathcal{V}$ is a countable local basis at $1_K$.
\end{proof}

We now present the proof of one of our main results. The proof follows the strategy of \cite[Theorem 12.8]{Bartle}, adapted to the present setting.

\begin{proof}[Proof of Theorem \ref{Thm:Hake}]
We will prove in detail the first part of the theorem and later briefly indicate the adaptations for the second part, since the proof is similar. Following the conclusion of the Lemma \ref{Lem:FirstCountableCompactLines}, we distinguish two cases. The first case is when there is a neighborhood $V$ of $1_K$ such that $G$ holds the same fixed value $\lambda$ on $V\setminus \{1_K\}$. 

Since $G$ is nondecreasing, we may fix $y\in V\setminus \{1_K\}$ such that $G$ s constant on $[y,1_K)$. By the definition of the integral and since $G$ is constant on $[y,1_K)$, for every $x\in [y,1_K]$,
\[\int_x^{1_K}f \, dG=f(x)G(x)+f(1_K)(G(1_K)-L_G(1_K)).\]
Since $f$ is $G$-integrable at $[0_K,y]$, it follows from \cite[Theorem 3.16]{CanKau} that $f$ is $G$-integrable in $K$. Now we have, again by \cite[Theorem 3.16]{CanKau}.
\begin{align*}
A&=\lim_{x \nearrow 1_K} \int_{0_K}^xf \, dG=\lim_{x \nearrow 1_K}\left(\int_K f \, dG-\int_x^{1_K}f \, dG+f(x)G(x)\right)\\
&=\int_K f \, dG-f(1_K)(G(1_K)-L_G(1_K)).
\end{align*}

In the second case, we may assume that $G$ is not constant on any neighborhood of $1_K$. Then, according to Lemma~\ref{Lem:FirstCountableCompactLines}, $1_K$ has a countable local basis, from which we can fix a strictly increasing sequence $(c_n)_n$ in $[0_K,1_K)$ such that $c_0=0_K$ and $\lim_{n\to\infty} c_n=1_K$.
Given $\varepsilon>0$, we choose $r\in\mathbb{N}$ such that
\[
L_G(1_K)-G(c_r)<\frac{\varepsilon}{|f(1_K)|+1}
\]
and such that, whenever $y \in (c_r,1_K)$,
\[\left|\int_{0_K}^y f \, dG-A\right|<\varepsilon.\]

We let $I_1=[c_0,c_1]$ and, for each $k>1$, we let $I_k=(c_{k-1},c_k]$. For each $k\ge1$, we define $f_{I_k}=f\cdot\chi_{I_k}$. Since $f$ is $G$-integrable on $[0_K,y]$ for every $y<1_K$, it follows from \cite[Proposition 3.14]{CanKau} that each $f_{I_k}$ is $G$-integrable, and we may fix $\delta_k$ to be an $\varepsilon/2^{k+1}$-gauge for $f_{I_k}$.

Without loss of generality, by shrinking the gauges if necessary and using the right-continuity of $G$, we may assume that $\delta_1(0_K)\subset [0_K,c_1)$ and that, for each $k\geq 1$:
\begin{itemize}
\item[(a)] $\delta_k(x)\subset (c_{k-1},c_k)$ if $x\in (c_{k-1},c_k)$;
\item[(b)] $\delta_k(c_k)\subset (c_{k-1},c_{k+1})$;
\item[(c)] whenever $x\in \delta_k(c_k)\cap [c_k,c_{k+1})$, then
\[|G(x)-G(c_k)|<\frac{\varepsilon}{2^{k+1}}.\]
\end{itemize}

We now define a gauge on $K$ by the formula
\[\delta(x) =
\begin{cases}
\delta_1(0_K), & \text{if } x=0_K,\\
\delta_k(x), & \text{if } x \in (c_{k-1},c_k],\\
(c_r,1_K], & \text{if } x=1_K.
\end{cases}\]

Let $P=\{([x_{i-1},x_i],t_i):i=1,\ldots,n\}$ be a $\delta$-fine partition of $K$.  We note that $1_K$ does not belong to any $I_k$, then $t_n=1_K$. Indeed, if $t_n<1_K$, then $t_n\in (c_{k-1},c_k]$ for some $k\ge1$, and hence $\delta(t_n)=\delta_k(t_n)\subset (c_{k-1},c_{k+1})\subset [0_K,1_K)$, which is impossible because $1_K\in(x_{n-1},x_n]\subset\delta(t_n)$. Therefore $t_n=1_K$. Moreover, since $(x_{n-1},x_n]\subset (c_r,1_K]$, we deduce that $c_r\leq x_{n-1}$. Let $s$ be the smallest integer such that $x_{n-1}\le c_s$.

By the construction of the gauge, for each $c_k$, with $1\leq k\leq s-1$, if $c_k\in (x_{i-1},x_i]$, then $t_i=c_k$. Indeed, since $k\le s-1$, we have $c_k\le c_{s-1}<x_{n-1}$, so $t_i\neq 1_K$. Hence $t_i\in (c_{j-1},c_j]$ for some $j\in\mathbb N$, and therefore $\delta(t_i)=\delta_j(t_i)$. Since $P$ is $\delta$-fine, $(x_{i-1},x_i]\subset \delta(t_i)$, and thus $c_k\in \delta_j(t_i)$. If $t_i\in (c_{j-1},c_j)$, then by (a) we have $\delta_j(t_i)\subset (c_{j-1},c_j)$, which is impossible. Therefore $t_i=c_j$. In this case, by (b), $\delta_j(c_j)\subset (c_{j-1},c_{j+1})$, and since $c_k\in \delta_j(c_j)$, the only possibility is $j=k$. Thus $t_i=c_k$.

By inserting additional points and tags, we may assume that $\{c_0,c_1,\ldots,c_{s-1}\}$ appear as both tags and division points of the partition; see \cite[Remark 3.6]{CanKau}.  Let $p_k$ and $q_k$ be such that $c_{k-1}=x_{p_k}$ and $c_k=x_{q_k}$, for each $1\leq k \leq s-1$. We note that the subintervals contained in $I_1=[c_0,c_1]$ and in $I_k=(c_{k-1},c_k]$, for $2\leq k \leq s-1$, correspond to the indices $i$ with $p_k+1\le i\le q_k$. Moreover, $q_k = p_{k+1}$ for each $1\leq k\leq s-2$, so that these intervals are adjacent and cover $[0_K,c_{s-1}]$ without overlap.

Since, for each $1\leq k \leq s-1$, $\delta_k$ is an $\varepsilon/2^{k+1}$-gauge for $f_{I_k}$, we have
\[
\left|f(0_K)G(0_K)+\sum_{i=p_1}^{q_1}f(t_i)(G(x_i)-G(x_{i-1}))+\Lambda-\int_K f_{I_1}\, dG\right|<\frac{\varepsilon}{2^2},
\]
where $\Lambda$ arises from the possible contribution of the interval containing $c_k$. More precisely, if $c_k$ is the tag of the interval $[c_k,x_{q_k+1}]$, then $\Lambda = f(c_k)(G(x_{q_k+1})-G(c_k))$, otherwise $\Lambda=0$. In any case, by condition (c), we have $|\Lambda|<\frac{\varepsilon}{2^2}$, and therefore
\[
\left|f(0_K)G(0_K)+\sum_{i=p_1}^{q_1}f(t_i)(G(x_i)-G(x_{i-1}))-\int_K f_{I_1}\, dG\right|<\frac{\varepsilon}{2}.
\]
A similar argument shows that, for each $1<k<s$,
\[
\left|\sum_{i=p_k}^{q_k}f_{I_k}(t_i)(G(x_i)-G(x_{i-1}))-\int_K f_{I_k}\, dG\right|<\frac{\varepsilon}{2^k}.
\]

Since $c_{s-1}<x_{n-1}\le c_s$, we consider the $\delta_s$-fine tagged system $\mathcal{S}=\{([x_{i-1},x_i],t_i): q_{s-1}+1\le i\le n-1\}$.
By the construction of the gauge, each of these intervals satisfies $(x_{i-1},x_i]\subset (c_{s-1},c_s)$, and hence $(x_{i-1},x_i]\subset \delta_s(t_i)$ for each $q_{s-1}+1\le i\le n-1$. Therefore, $\mathcal{S}$ is $\delta_s$-fine. We may therefore apply the Saks--Henstock Lemma \cite[Lemma 3.18]{CanKau} to obtain
\begin{equation}\label{Rel:Aux1}
\left|
\sum_{i=q_{s-1}+1}^{n-1}
\left(
f(t_i)(G(x_i)-G(x_{i-1}))
+ f(x_{i-1})G(x_{i-1})
- \int_{x_{i-1}}^{x_i} f\,dG
\right)
\right|
\le \frac{\varepsilon}{2^{s}}.
\end{equation}

By \cite[Propositions 3.12 and 3.15]{CanKau}, we may write, for each $q_{s-1}+1\leq i\leq n-1$,
\begin{align*}
\int_{x_{i-1}}^{x_i} f \, dG & - f(x_{i-1}) G(x_{i-1})= \left(\int_{x_{i-1}}^{x_i} f \, dG - f(x_{i-1}) L_G(x_{i-1}) \right) \\
& \quad - f(x_{i-1}) \left( G(x_{i-1}) - L_G(x_{i-1}) \right)= \int_K f|_{(x_{i-1},x_i]} \, dG.
\end{align*}
We can therefore rewrite \eqref{Rel:Aux1} as
\[
\left|
\sum_{i=q_{s-1}+1}^{n-1} f(t_i)(G(x_i)-G(x_{i-1}))
- \int_K f|_{(c_{s-1},x_{n-1}]} \, dG
\right|
\leq \frac{\varepsilon}{2^{s}}.
\]

Now we gather all the information above, and recalling that $G$ is nondecreasing:

\begin{align*}
&\left|S(f,G,P)-(A+f(1_K)(G(1_K)-L_G(1_K)))\right|\\
&=\left|f(0_K)G(0_K)+\sum_{i=1}^n f(t_i)(G(x_i)-G(x_{i-1}))-(A+f(1_K)(G(1_K)-L_G(1_K)))\right|\\
&=\left|f(0_K)G(0_K)
+\sum_{k=1}^{s-1}\left(\sum_{i=p_k}^{q_k} f(t_i)(G(x_i)-G(x_{i-1}))\right)
+\sum_{i=q_{s-1}+1}^{n-1} f(t_i)(G(x_i)-G(x_{i-1}))\right.\\
&\quad \left.
+f(1_K)(G(1_K)-G(x_{n-1}))
-(A+f(1_K)(G(1_K)-L_G(1_K)))\right|\\
&<\left|\sum_{k=1}^{s-1}\int_K f_{I_k}\,dG
+\int_K f|_{(c_{s-1},x_{n-1}]}\, dG
+f(1_K)(L_G(1_K)-G(x_{n-1}))
-A\right|+2\varepsilon\\
&\leq \left|\int_{0_K}^{x_{n-1}} f\,dG-A\right|
+|f(1_K)|(L_G(1_K)-G(x_{n-1}))
+2\varepsilon\\
&<4\varepsilon.
\end{align*}
And we are done with the first part.

The second part follows by adapting the proof of the first part. We briefly indicate the required modifications. We consider a decreasing sequence $(c_n)_n$ converging to $0_K$ and intervals of the form $(c_k,c_{k-1}]$. The construction of the gauges and the refinement of the partition are analogous, with the roles of $0_K$ and $1_K$ interchanged.

The only additional care is needed near $0_K$, where the term $f(0_K)G(0_K)$ appears due to the definition of the integral. The estimates are handled similarly to those in the first part, using the right-continuity of $G$ at $0_K$ and an application of the Saks--Henstock Lemma \cite[Lemma 3.18]{CanKau}.
\end{proof}

\begin{remark}
As an application of the theorem, we illustrate the necessity of the additional term $f(1_K)(G(1_K)-L_G(1_K))$. Consider the following example. Let $K=[0,1]$ and define $f,G:K\to \mathbb{R}$ by
\[
f(x)=G(x)=
\begin{cases}
0, & \text{if } x<1,\\
1, & \text{if } x=1.
\end{cases}
\]
Then
\[\int_K f\, dG = f(1_K)(G(1_K)-L_G(1_K))=1,\]
while
\[\lim_{y \nearrow 1_K} \int_{0_K}^y f \, dG=0.\]
\end{remark}

To conclude this section, we present the converse of Theorem~\ref{Thm:Hake}. We first need the following simple lemma.

\begin{lemma}\label{Lem:Continuity0_K}
Let $K$ be a compact line, and let $G:K\to\mathbb R$ be an amenable function. Then $\lim_{y\searrow 0_K}L_G(y)=G(0_K)$.
\end{lemma}
\begin{proof}
Let $\varepsilon>0$ be arbitrary. Since $G$ is right-continuous at $0_K$, there exists $c\in(0_K,1_K)$ such that $|G(t)-G(0_K)|<\varepsilon$ for every $t\in(0_K,c]$. Now let $y\in(0_K,c]$ be arbitrary. If $y$ is left-isolated, then $L_G(y)=G(y^-)$. Since $0_K<y^-\le y\le c$, we have $y^-\in(0_K,c]$, and therefore
\[|L_G(y)-G(0_K)|=|G(y^-)-G(0_K)|<\varepsilon.\]

If $y$ is left-dense, then by definition $L_G(y)=\lim_{t\nearrow y}G(t)$. For every $t\in(0_K,y)$ we have $t\in(0_K,c]$, hence $|G(t)-G(0_K)|<\varepsilon.$ Passing to the limit as $t\nearrow y$, we obtain
\[
|L_G(y)-G(0_K)|\le \varepsilon.
\]
Since $\varepsilon>0$ is arbitrary, we deduce that $\lim_{y\searrow 0_K}L_G(y)=G(0_K)$.
\end{proof}

\begin{theorem}\label{Thm:HakeConverse}
Let $K$ be a compact line and let $G:K\to\mathbb R$ be a nondecreasing right-continuous function. If $1_K$ is left-dense and $f$ is $G$-integrable on $K$, then
\[
\lim_{y\nearrow 1_K}\int_K f_{[0_K,y]}\,dG
=
\int_K f\,dG-f(1_K)(G(1_K)-L_G(1_K)).
\]

If $0_K$ is right-dense and $f$ is $G$-integrable on $K$, then
\[
\lim_{y \searrow 0_K} \int_K f_{(y,1_K]} \, dG
=
\int_K f\,dG-f(0_K)G(0_K).
\]
\end{theorem}
\begin{proof}
We first notice that, by \cite[Theorem 3.15]{CanKau}, the assumption that $f$ is $G$-integrable implies that $f_{[0_K,y]}$ and $f_{[y,1_K]}$ are $G$-integrable for every $y\in K$.

For the first part, we apply \cite[Theorem 4.1]{CanKau} to the function $F(x)=\int_{0_K}^x f\,dG$ with $c=1_K$ to obtain
\[
F(1_K)-L_F(1_K)=f(1_K)(G(1_K)-L_G(1_K)).
\]
In other words,
\[
\lim_{y\nearrow 1_K}\int_K f_{[0_K,y]}\,dG=L_F(1_K)=\int_K f\,dG-f(1_K)(G(1_K)-L_G(1_K)).
\]

On the other hand, again by \cite[Theorem 4.1]{CanKau}, this time with $c=y\in(0_K,1_K)$, we obtain
\[F(y)-L_F(y)=f(y)(G(y)-L_G(y)).\]
By additivity on intervals and \cite[Proposition 3.15]{CanKau}, we have
\begin{align*}
F(y)&=\int_{0_K}^y f\,dG=\int_K f_{[0_K,y]}\,dG
=\int_K f\,dG-\int_K f_{(y,1_K]}\,dG.
\end{align*}
Whence
\[
\int_K f_{(y,1_K]}\,dG=\int_K f\,dG-L_F(y)-f(y)(G(y)-L_G(y)).
\]

Taking the limit as $y\searrow 0_K$, and noticing that, since $F$ (see \cite[Theorem 4.1]{CanKau}) and $G$ are amenable, we may apply Lemma~\ref{Lem:Continuity0_K} to obtain
\[
\lim_{y\searrow 0_K}L_F(y)=L_F(0_K)=f(0_K)G(0_K),\qquad \lim_{y\searrow 0_K}(G(y)-L_G(y))=0,
\]
which leads to
\[
\lim_{y \searrow 0_K} \int_K f_{(y,1_K]} \, dG=\int_K f\,dG-f(0_K)G(0_K).
\]
completing the proof.
\end{proof}

\section{Measurability for $\mathrm{NBV}(K)$ Integrators}
\label{Sec:MeasurableNBV}
In this section, we extend the measurability results obtained in Section~\ref{Sec:Measurable} to integrators in $\mathrm{NBV}(K)$. We begin with the following definition, which will play an important role in this study.

\begin{defin}\label{Def:VariationFunction}
Let $K$ be a compact line and let $G\in \mathrm{NBV}(K)$. The \emph{variation function} of $G$ is the function $T_G:K\to\mathbb{R}$ defined by
\[
T_G(x)=|G(0_K)|+\operatorname{Var}(G|_{[0_K,x]}),
\]
where $V$ denotes the total variation (see Definition~\ref{Def:Variation}).
\end{defin}

The following proposition and Theorem~\ref{Thm:VariationMeasure} can be derived from results in \cite[Corollary~2.4.7]{Ronchim}. For the convenience of the reader, we present complete proofs adapted to our setting.

\begin{prop}\label{Prop:VariationDoneRight}
Let $K$ be a compact line and let $G\in \mathrm{NBV}(K)$. Let $T_G$ be as in Definition~\ref{Def:VariationFunction}. Then $T_G$ is nondecreasing and right-continuous. Moreover, setting $G_1=T_G$ and $G_2=T_G-G$, we have that $G_1$ and $G_2$ are nondecreasing and right-continuous, and
\[G=G_1-G_2.\]
\end{prop}
\begin{proof}
It follows from the definition of total variation that $T_G$ is increasing. To prove that $T_G$ is right-continuous, let $a \in K$ be arbitrary. We may suppose that $a$ is not right-isolated as any function is right-continuous at a right-isolated point. Let $\ell=\inf\{T_G(y):y>a\}$. Since $T_G$ is nondecreasingand  and $G$ is right-continuous, there exists $v>t$ such that 
\[\ell\leq T_G(y)<\ell+\frac{\varepsilon}{2} \text{ and } |G(y)-f(a)|<\frac{\varepsilon}{2}\]
for all $y \in (a,v)$.
Given $y \in (a,v)$, let $D=\{x_0,x_1,\ldots,x_n\}$ be a division of $[0,y]$ and let $r$ be the smallest index such that $x_{r-1}\leq a<x_r$. Then we have:
\begin{align*}
\operatorname{Var}(G|_{[0_K,y]},D)&= \sum_{i=1}^{r-1}|G(x_i)-G(x_{i-1})| + |G(x_r)-G(x_{r-1})| + \sum_{j=r+1}^n|G(x_j)-G(x_{j-1})|\\
               &\leq (T_G(x_{r-1})-|G(0_K)|)+|G(x_r)-G(a)|+(T_G(a)-T_G(x_{r-1})) +\operatorname{Var}(G|_{[x_r,y]})\\
               &< T_G(a)-|G(0_K)|+\frac{\varepsilon}{2}+(T_G(y)-T_G(x_r))\\
               &<T_G(a)-|G(0_K)|+\frac{\varepsilon}{2}+(T_G(y)-\ell)\\
               &<T_G(a)-|G(0_K)|+\varepsilon.
\end{align*}
Taking the supremum over all divisions $D$, we conclude that  $T_G(y)\le T_G(a)+\varepsilon$
for all $y\in(a,v)$. We deduce that $|T_G(y)-T_G(a)|<\varepsilon$ for all $y \in (a,v)$ which proves that $T_G$ is right-continuous at $a$.

Next, by defining $G_1=T_G$ and $G_2=T_G-G$, it is clear tat $G_2$ is also right-continuous and $G=G_1-G_2$. Moreover, for every $x, y \in K$, if $y \leq x$, then $G(x)-G(y)\leq T_G(x)-T_G(y)$. Therefore
\[G_2(y)=T_G(y)-G(y)\leq T_G(x)-G(x)=G_2(x),\]
which proves that $G_2$ is nondecreasing.
\end{proof}

\begin{remark}\label{Rem:MeasureInterval}
In order to establish the next result, we need a remark concerning the representation of open intervals in a compact line $K$. Here, by an open interval we mean an interval in the topological sense, that is, an interval which is also an open subset of $K$.

In general, a given open interval in $K$ may admit more than one representation. For instance, if $b$ is right-isolated and $a$ is left-isolated, then the sets $(a^-,b^+)$ and $[a,b]$ determine the same open subset of $K$. Although these representations differ, they yield the same value under the measure $\mu_G$, whenever $G\in \mathrm{NBV}(K)$. Indeed,
\[
\mu_G((a^-,b^+))=L_G(b^+)-G(a^-)=G(b)-G(a^-)=G(b)-L_G(a)=\mu_G([a,b]).
\]

However, in the arguments that follow, it will be important to fix a convenient representation for intervals, in order to avoid overlaps at the endpoints of consecutive disjoint intervals. For example, let $a<b<c$ be points of $K$. If $b$ is right-isolated, then $(a,b^+)$ and $(b,c)$ are disjoint, but the right endpoint of the former exceeds the left endpoint of the latter. Such a situation may compromise estimates involving the variation in Theorem~\ref{Thm:VariationMeasure} below.

For this reason, it will be convenient to fix a specific representation for open intervals in $K$.

Let $I$ be a non-empty open interval, represented with endpoints $a\le b$. We define a new representation of $I$, called the canonical representation, as follows:
\begin{itemize}
\item on the left:
\begin{itemize}
\item if $a\notin I$, we write $(a,\cdot)$;
\item if $a\in I$ and $a>0_K$, we write $(a^-,\cdot)$;
\item if $a=0_K$, we write $[0_K,\cdot)$;
\end{itemize}

\item on the right:
\begin{itemize}
\item if $b\in I$, we write $(\cdot,b]$;
\item if $b\notin I$ and $b$ is left-isolated, we write $(\cdot,b^-]$;
\item if $b\notin I$ and $b$ is not left-isolated, we write $(\cdot,b)$.
\end{itemize}
\end{itemize}

With this representation, no overlap occurs at the endpoints of disjoint open intervals. Indeed, let $I_1<I_2$ be two disjoint open intervals represented in canonical form, with endpoints $a_1\le b_1$ and $a_2\le b_2$, respectively. If $a_2<b_1$, then the rules above imply that $a_2$ is right-isolated and $a_2^+\in I_2$, while $b_1$ is left-isolated and $b_1\in I_1$. This implies that $a_2^+\in I_1\cap I_2$, which is a contradiction.
\end{remark}

\begin{prop}\label{Prop:OverlapNoWay}
Let $K$ be a compact line and let $G\in \mathrm{NBV}(K)$. Let $I$ be an open interval of $K$, and let $a\le b$ be its endpoints in the canonical representation. Then
\[
|\mu_G(I)|\le \operatorname{Var}(G|_{[a,b]}).
\]
\end{prop}

\begin{proof}
If $I$ is either $(a,b]$ or $[0_K,b]$, then the result is immediate, since $|\mu_G(I)|=|G(b)-G(a)|$. It remains to consider the case where $I$ has the canonical representation $(a,b)$ with $b$ left-dense. In this case, $|\mu_G(I)|=|L_G(b)-G(a)|$.

For every $\varepsilon>0$, there exists $x_\varepsilon<b$ such that $L_G(b)-\varepsilon<G(x_\varepsilon)\le L_G(b)$.
Hence
\begin{align*}
|G(a)-L_G(b)|
&\le |G(a)-G(x_\varepsilon)|+|G(x_\varepsilon)-L_G(b)|\\
&< |G(a)-G(x_\varepsilon)|+\varepsilon\le \operatorname{Var}(G|_{[a,b]})+\varepsilon.
\end{align*}
Letting $\varepsilon\to 0$, we obtain
\[
|\mu_G(I)|=|G(a)-L_G(b)|\le \operatorname{Var}(G|_{[a,b]}).
\]
\end{proof}

\begin{theorem}\label{Thm:VariationMeasure}
Let $K$ be a compact line and let $G\in \mathrm{NBV}(K)$. Then $|\mu_G|=\mu_{T_G}$.
\end{theorem}
\begin{proof}
According to \cite[Theorem 2.4.11]{Ronchim}, we will have the thesis by proving that for every $x\in K$, we have $|\mu_G|([0_K,x])=\mu_{T_G}([0_K,x])$.

On the one hand, by the definition of the total variation measure, we have, for every $x\in K$,
\[
\mu_{T_G}([0_K,x])=T_G(x)-T_G(0_K)=\operatorname{Var}(G|_{[0_K,x]})\leq |\mu_G|([0_K,x]).
\]

For the opposite inequality, let $x\in K$, $\varepsilon>0$ be arbitrary and let $\{E_1,\ldots,E_n\}$ be a finite Borel partition of $[0_K,x]$ such that 
\[
|\mu_G|([0_K,x])<\sum_{j=1}^n|\mu_G(E_j)|+\varepsilon.
\]

Since $|\mu_G|$ is a Radon measure, by regularity, for each $j$ we may choose a compact set 
$K_j\subseteq E_j$ such that 
\[
|\mu_G|(E_j\setminus K_j)<\frac{\varepsilon}{n}.
\]
Then the sets $K_1,\dots,K_n$ are pairwise disjoint and
\[
\sum_{j=1}^n |\mu_G(E_j)|
\le
\sum_{j=1}^n |\mu_G(K_j)|+\varepsilon.
\]

By the normality of $K$ and again by the regularity of $|\mu_G|$, we may fix pairwise disjoint open sets $V_1,\dots,V_n$ such that $V_1<\cdots<V_n$, with $K_j\subseteq V_j$ and 
\[
|\mu_G|(V_j\setminus K_j)<\frac{\varepsilon}{n}
\]
for all $j$.

Since $K$ is a compact line, each open set $V_j$ can be written as a union of pairwise disjoint open intervals. Then, by compactness of $K_j$, there exist pairwise disjoint open intervals $I_1^j<\cdots<I_{m_j}^j$ such that 
\[
K_j\subset \bigcup_{i=1}^{m_j} I_{i}^j\subset V_j.
\]
Writing $V_j^\prime=\bigcup_{i=1}^{m_j} I_i^j$, we obtain
\[
|\mu_G|(V_j^\prime\setminus K_j)\leq |\mu_G|(V_j\setminus K_j)<\frac{\varepsilon}{n}.
\]
Thus,
\[
\sum_{j=1}^n |\mu_G(E_j)|\le\sum_{j=1}^n |\mu_G(V_j^\prime)|+2\varepsilon.
\]

Now, for each $j$, since $V_j^\prime=\bigcup_{i=1}^{m_j} I_i^j$, by letting $a_i^j\leq b_i^j$ denote the endpoints of the interval $I_i^j$ represented in canonical form as in Remark \ref{Rem:MeasureInterval}, and using Proposition \ref{Prop:OverlapNoWay}, we obtain
\[
|\mu_G(V_j^\prime)|\le\sum_{i=1}^{m_j} |\mu_G(I_i^j)|\le\sum_{i=1}^{m_j} \operatorname{Var}(G|_{[a_i^j,b_i^j]})\leq \sum_{i=1}^{m_j}(T_G(b_i^j)-T_G(a_i^j)).
\]

Now, since the intervals are represented in canonical form, we have that $T_G(a_i^j)\geq T_G(b_{i-1}^j)$ for each $2\leq i \leq m_j$, and hence
\[
|\mu_G(V_j^\prime)|\le T_G(b_{m_j}^j)-T_G(a_1^j).
\]

By the same reasoning, we have $T_G(a_1^j)\geq T_G(b_{m_{j-1}}^{j-1})$ for each $2\leq j \leq n$, which gives
\[
\sum_{j=1}^n |\mu_G(V_j^\prime)|\le\sum_{j=1}^n(T_G(b_{m_j}^j)-T_G(a_1^j))\le T_G(b_{m_n}^n)-T_G(a_1^1)\leq T_G(x)-T_G(0_K)=\mu_{T_G}([0_K,x]).
\]

Therefore,
\[
|\mu_G|([0_K,x])\le\mu_{T_G}([0_K,x])+3\varepsilon.
\]
Letting $\varepsilon\to 0$, we obtain the desired inequality.
\end{proof}

\begin{prop}\label{Prop:TIntegrableImpliesGIntegrable} 
Let $K$ be a compact line and let $G\in \mathrm{NBV}(K)$. If $f$ and $|f|$ are both $T_G$-integrable, then they are also $(T_G-G)$-integrable, and consequently $G$-integrable.
\end{prop}
\begin{proof}
By Proposition~\ref{Prop:VariationDoneRight}, by defining $G_1=T_G$ and $G_2=T_G-G$, then $G_1$ and $G_2$ are nondecreasing and right-continuous and $G=G_1-G_2$.

Since $f$ and $|f|$ are both $T_G$-integrable. By Corollary~\ref{Cor:CharacterizationIncreasingMeasurability}, it follows that $f\in L^1(\mu_{T_G})$.

From Theorem \ref{Thm:VariationMeasure}, we know that $|\mu_G|=\mu_{T_G}$. Therefore, for each interval $I\subset K$,
\[\mu_{G_2}(I)\le |\mu_G|(I)=\mu_{T_G}(I).\]
Since the intervals generate the Borel $\sigma$-algebra of $K$ it follows that $\mu_{G_2}(A)\le \mu_{T_G}(A)$ for every Borel set $A\subset K$. Since $\mu_{G_2}$ and $\mu_{T_G}$ are positive measures, we obtain that if $f\in L^1(\mu_{T_G})$, then $f\in L^1(\mu_{G_2})$.

Applying Corollary~\ref{Cor:CharacterizationIncreasingMeasurability} to $G_2$ allows us to conclude that both $f$ and $|f|$ are $G_2$-integrable. Finally, since $G=G_1-G_2$,  \cite[Proposition~3.11]{CanKau} implies that both $f$ and $|f|$ are $G$-integrable.
\end{proof}

The next example shows that the converse of Proposition~\ref{Prop:TIntegrableImpliesGIntegrable} fails in general. Indeed, if $f$ is $G$-integrable, where $G\in \mathrm{NBV}(K)$ is not monotone and admits a decomposition $G=G_1-G_2$ as in Proposition~\ref{Prop:VariationDoneRight}, with $G_1$ and $G_2$ non-null, it may happen that $f$ is neither $G_1$-integrable nor $G_2$-integrable.

\begin{example}
There exist a compact line $K$, a function $G\in\mathrm{NBV}(K)$, and a function $f:K\to\mathbb R$ such that $f$ is $G$-integrable but not $T_G$-integrable. Indeed, let $K=[0,1]$ and define $G:[0,1]\to \mathbb{R}$ by
\[
G(x)=
\begin{cases}
\displaystyle \int_0^x t\sin\!\left(\frac1{t^2}\right)\,dt, & \text{ if }x>0,\\[1.2ex]
0, & \text{ if }x=0.
\end{cases}
\]
Since 
\[\int_0^1 \left|x\sin\!\left(\frac1{x^2}\right)\right|\,dx<\infty.\]
we obtain from \cite[Theorem 4.2]{CanKau} that $G$ is an element of $\mathrm{NBV}(K)$ and 
\[T_G(x)=\operatorname{Var}(G|_{[0,x]})+|G(0)|=\int_0^x \left|t\sin\!\left(\frac1{t^2}\right)\right|\,dt.\]

Let $f:[0,1]\to \mathbb{R}$ be given by the formula
\[
f(x)=
\begin{cases}
\displaystyle \frac1{x^2}, & x>0,\\[1ex]
0, & x=0.
\end{cases}
\]

We claim that $f$ is $G$-integrable but not $T_G$-integrable. Indeed, let $y\in(0,1)$ be arbitrary. Since $f$ is continuous on $[y,1]$, it follows that $f$ is $G$-integrable on $[y,1]$, and
\[
\int_K f|_{(y,1]}\,dG=\int_y^1 \frac1{x^2}\,x\sin\!\left(\frac1{x^2}\right)\,dx=\int_y^1 \frac{\sin(1/x^2)}{x}\,dx=\frac12\int_1^{1/y^2}\frac{\sin u}{u}\,du.
\]
Hence 
\[\lim_{y\searrow 0}\int_K f|_{(y,1]}\,dG=\frac12\int_1^\infty \frac{\sin u}{u}\,du\in \mathbb{R}\]
and we deduce from Hake theorem \ref{Thm:Hake} that $f$ is $G$-integrable.

On the other hand, for each $y \in (0,1)$ holds
\[
\int_K f|_{(y,1]}\,dT_G=\int_y^1 \frac1{x^2}\left|x\sin\!\left(\frac1{x^2}\right)\right|\,dx=\int_y^1 \frac{|\sin(1/x^2)|}{x}\,dx=
\frac12\int_1^{1/y^2}\frac{|\sin u|}{u}\,du.
\]
Hence 
\[\lim_{y\searrow 0}\int_K f|_{(y,1]}\,dT_G=\frac12\int_1^\infty \left|\frac{\sin u}{u}\right|\,du=+\infty\]
and we deduce from the converse of Hake theorem \ref{Thm:HakeConverse} that $f$ is not $T_G$-integrable.
\end{example}

The following substitution lemma provides a useful tool for transferring results from the case of a nondecreasing integrator to the more general setting of integrators in $\mathrm{NBV}(K)$. Its main limitation is the boundedness assumption on the integrand $f$. At present, we do not know whether a corresponding general result can be obtained without this hypothesis.

\begin{lemma}[Substitution]\label{Lem:Substitution}
Let $K$ be a compact line, let $T:K\to\mathbb R$ be a nondecreasing right-continuous function, and let $h:K\to\mathbb R$ be a $T$-integrable function. Define
\[
H(x)=\int_{0_K}^x h\,dT,\qquad x\in K.
\]
If $f:K\to\mathbb R$ is bounded, then $f$ is $H$-integrable if and only if $fh$ is $T$-integrable. In this case,
\[
\int_K f\,dH=\int_K fh\,dT.
\]
\end{lemma}
\begin{proof}
Let $\varepsilon>0$ be arbitrary and let $\sup_{x\in K}|f(x)|\le M$. If $P=\{([x_{i-1},x_i],t_i):i=1,\ldots,n\}$ is any tagged partition of $K$, recalling \cite[Theorem 3.16]{CanKau} we can write
\begin{align*}S(f,H,P)&=f(0_K)H(0_K)+\sum_{i=1}^n
f(t_i)\bigl(H(x_i)-H(x_{i-1})\bigr)\\
&=
f(0_K)h(0_K)T(0_K)+\sum_{i=1}^nf(t_i)\left(\int_{x_{i-1}}^{x_i} h\,dT-f(x_{i-1})T(x_{i-1})\right).
\end{align*}
and
\[
S(fh,T,P)=
f(0_K)h(0_K)T(0_K)+\sum_{i=1}^nf(t_i)h(t_i)\bigl(T(x_i)-T(x_{i-1})\bigr).
\]
Therefore
\begin{align*}
|S(f,H,P)-S(fh,T,P)|
&=
\left|
\sum_{i=1}^n
f(t_i)\bigl(h(t_i)(T(x_i)-T(x_{i-1})+h(x_{i-1})T(x_{i-1}) -\int_{x_{i-1}}^{x_i} h\,dT\bigr)\right|\\
&\leq 
M\sum_{i=1}^n\left|\bigl(h(t_i)(T(x_i)-T(x_{i-1})+h(x_{i-1})T(x_{i-1}) -\int_{x_{i-1}}^{x_i} h\,dT\bigr)\right|
\end{align*}

Since $h$ is $T$-integrable, we may fix an $\frac{\varepsilon}{8(M+1)}$-gauge $\delta_h$ for $h$. Then for any $\delta_h$-fine 
partition $P=\{([x_{i-1},x_i],t_i):i=1,\ldots,n\}$ of $K$, by employing the Saks--Henstock Lemma \cite[Corollary 3.19]{CanKau} for the tagged system obtained from $P$, we obtain
\[
\sum_{i=1}^n\left|h(t_i)\bigl(T(x_i)-T(x_{i-1})+h(x_{i-1})T(x_{i-1}) -\int_{x_{i-1}}^{x_i} h\,dT\bigr)\right|<\frac{\varepsilon}{2(M+1)}.
\]
and we deduce that for each $\delta_h$-fine partition $P$ of $K$ holds
\[|S(f,H,P)-S(fh,T,P)|\leq M\frac{\varepsilon}{2(M+1)}<\frac{\varepsilon}{2}.\]

If $f$ is $H$-integrable, we fix a $\frac{\epsilon}{2}$-gauge $\delta_f$ for $K$ and let $P$ be a $\delta_f\cap \delta_h$-fine partition of $K$. We have
\begin{align*}
\left|S(fh,T,P)-\int_K f\,dH\right|
&\le |S(fh,T,P)-S(f,H,P)|+\left|S(f,H,P)-\int_K f\,dH\right|\\
&<\frac{\varepsilon}{2}+\frac{\varepsilon}{8(M+1)}
<\varepsilon.
\end{align*}

Hence $fh$ is $T$-integrable and
\[
\int_K fh\,dT=\int_K f\,dH.
\]
On the other hand, if $fh$ is $T$-integrable, we let we fix a $\frac{\epsilon}{2}$-gauge $\delta_{fh}$ for $fh$ and let $P$ be a $\delta_{fh}\cap \delta_h$-fine partition of $K$. It follows that
\begin{align*}
\left|S(f,H,P)-\int_K fh\,dT\right|
&\le |S(f,H,P)-S(fh,T,P)|+\left|S(fh,T,P)-\int_K fh\,dT\right|\\
&<\frac{\varepsilon}{2}+\frac{\varepsilon}{8(M+1)}
<\varepsilon.
\end{align*}
Hence $f$ is $H$-integrable and
\[
\int_K f\,dH=\int_K fh\,dT.
\]
This concludes the proof.
\end{proof}

We are now in a position to establish another one of the main results of this paper.

\begin{proof}[Proof of Theorem \ref{Thm:Main2}]
According to Theorem~\ref{Thm:VariationMeasure}, we have $|\mu_G|=\mu_{T_G}$. By the Radon-Nikod\'ym theorem \cite[\S 31, Theorem B]{Halmos}, there exists a function $h\in L^1(\mu_{T_G})$ such that
\[
\mu_G(A)=\int_A h\,d\mu_{T_G}
\]
for every Borel set $A\subset K$. Moreover, since $|\mu_G|=\mu_{T_G}$, it follows that $|h|=1$ $\mu_{T_G}$-almost everywhere.

Define $H:K\to\mathbb R$ by
\[
H(x)=\int_{0_K}^x h\,dT_G.
\]

Since $h\in L^1(\mu_{T_G})$ and $T_G$ is nondecreasing and right-continuous, Corollary~\ref{Cor:CharacterizationIncreasingMeasurability} implies that $h$ is $T_G$-integrable. Hence, by \cite[Theorem~4.2]{CanKau}, $H\in \mathrm{NBV}(K)$.

We claim that $H=G$. Indeed, for every $x\in K$,
\[
G(x)=\mu_G([0_K,x])=\int_{[0_K,x]} h\,d\mu_{T_G}=\int_{0_K}^x h\,dT_G=H(x).
\]

Next, let $f:K\to\mathbb R$ be $G$-integrable and bounded. Since $H=G$, by Lemma~\ref{Lem:Substitution} applied to $H$, we obtain that $fh$ is $T_G$-integrable and
\[
\int_K f\,dG=\int_K f\,dH=\int_K fh\,dT_G.
\]

By Theorem~\ref{Thm:Main1}, it follows that $fh$ is $\mu_{T_G}$-measurable. Since $h$ is $\mu_{T_G}$-measurable and $|h|=1$ $\mu_{T_G}$-almost everywhere, we have $f=(fh)h$ $\mu_{T_G}$-almost everywhere.

Therefore, $f$ is $\mu_{T_G}$-measurable. Since $\mu_{T_G}=|\mu_G|$ by Theorem~\ref{Thm:VariationMeasure}, we conclude that $f$ is $\mu_G$-measurable.

For the second part, assume that both $f$ and $|f|$ are $G$-integrable. For every $m\in\mathbb N$, the functions $f-m$ and $|f|+m$ are $G$-integrable. Hence, since $H=G$, Lemma~\ref{Lem:Substitution} implies that both $(f-m)h$ and $(|f|+m)h$ are $T_G$-integrable. Moreover,
\[
|(f-m)h|=|f-m|\le |f|+m = |(|f|+m)h|
\]
$\mu_{T_G}$-almost everywhere. Therefore, by \cite[Corollary~4.3]{CanKau}, the function $|f-m|$ is $T_G$-integrable. By Proposition~\ref{Prop:TIntegrableImpliesGIntegrable}, it follows that $|f-m|$ is $G$-integrable.

As a consequence, the function $f_m:K\to\mathbb R$ defined by
\[
f_m(x)=\max\{-m,\min\{f(x),m\}\}
=\frac{f(x)+m-|f(x)-m|}{2}
\]
is $G$-integrable. Since $f_m$ is bounded, the first part implies that $f_m$ is $\mu_G$-measurable for every $m\in\mathbb N$. Since $(f_m)_m$ converges pointwise to $f$, we conclude that $f$ is $\mu_G$-measurable.
\end{proof}

\begin{remark}
We do not know whether Theorem~\ref{Thm:Main2} holds in general for all $G$-integrable functions. Although we were not able to provide a proof or a counterexample, we believe that such a result should be true.
\end{remark}

To conclude the paper, we present an $\mathrm{NBV}(K)$ analogue of Corollary~\ref{Cor:CharacterizationIncreasingMeasurability}.

\begin{corollary}\label{Cor:CharacterizationMeasurability}
Let $K$ be a compact line, and let $G\in \mathrm{NBV}(K)$. Then, for every function $f:K\to\mathbb R$, the following are equivalent:
\begin{enumerate}
    \item $f\in L_1(\mu_G)$;
    \item both $f$ and $|f|$ are $T_G$-integrable.
\end{enumerate}
In this case, $f$ is $G$-integrable and
\[
\int_K f\,dG=\int_K f\,d\mu_G.
\]
\end{corollary}
\begin{proof}
Since $|\mu_G|=\mu_{T_G}$, the equivalence between (1) and (2) follows from Corollary~\ref{Cor:CharacterizationIncreasingMeasurability}. 

To verify the identity between the integrals, note that if $f$ and $|f|$ are both $T_G$-integrable, then by Proposition~\ref{Prop:TIntegrableImpliesGIntegrable}, they are also $(T_G-G)$-integrable. Letting $G_1=T_G$ and $G_2=T_G-G$, Proposition~\ref{Prop:VariationDoneRight} ensures that both $G_1$ and $G_2$ are nondecreasing and right-continuous.

Applying Corollary~\ref{Cor:CharacterizationIncreasingMeasurability} once more, we obtain $f\in L_1(\mu_{G_1})\cap L_1(\mu_{G_2})$, and hence
\[
\int_K f\, d\mu_G=\int_K f\, d\mu_{G_1}-\int_K f\, d\mu_{G_2}=\int_K f\, dG_1-\int_K f\, dG_2=\int_K f\, dG.
\]
\end{proof}


\begin{thebibliography}{99}

\bibitem{Bartle}
R. G. Bartle,
\emph{A Modern Theory of Integration},
Grad. Stud. Math., Vol. 32,
Amer. Math. Soc., Providence, RI, 2001.

\bibitem{BohnerPeterson}
M. Bohner and A. Peterson,
\emph{Dynamic Equations on Time Scales: An Introduction with Applications},
Model. Simul. Sci. Eng. Technol.,
Birkh\"auser Boston, Boston, MA, 2001.
doi:10.1007/978-1-4612-0201-1.

\bibitem{BianconiKaufmann}
R. Bianconi and P. L. Kaufmann,
\emph{Triangle integral---a nonabsolute integration process suitable for piecewise linear surfaces},
Real Anal. Exchange
36 (2010), no. 2, 373--404.

\bibitem{CanKau}
L. Candido and P. L. Kaufmann,
\emph{Kurzweil--Stieltjes integration on compact lines},
Positivity
30 (2026), Art.~6.
doi:10.1007/s11117-025-01161-9.

\bibitem{Engelking}
R. Engelking,
\emph{General Topology},
Sigma Ser. Pure Math., Vol. 6,
Heldermann, Berlin, 1989.

\bibitem{Halmos}
P. R. Halmos,
\emph{Measure Theory},
Grad. Texts in Math., Vol. 18,
Springer, New York, 1974.

\bibitem{Henstock1968}
R. Henstock,
\emph{A Riemann-type integral of Lebesgue power},
Can. J. Math.
20 (1968), 79--87.

\bibitem{KurtzSwartz}
D. S. Kurtz and C. W. Swartz,
\emph{Theories of Integration: The Integrals of Riemann, Lebesgue, Henstock--Kurzweil, and McShane},
Ser. Real Anal., Vol. 9,
World Scientific, Singapore, 2004.

\bibitem{Kurzweil1957}
J. Kurzweil,
\emph{Generalized ordinary differential equations and continuous dependence on a parameter},
Czechoslovak Math. J.
7 (1957), no. 3, 418--449.

\bibitem{MST}
G. A. Monteiro, A. Slav\'ik, and M. Tvrd\'y,
\emph{Kurzweil--Stieltjes Integral: Theory and Applications},
Ser. Real Anal., Vol. 17,
World Scientific, Singapore, 2019.
doi:10.1142/11294.

\bibitem{Lee}
T. Y. Lee,
\emph{Henstock--Kurzweil Integration on Euclidean Spaces},
Ser. Real Anal., Vol. 12,
World Scientific, Singapore, 2011.
doi:10.1142/7881.

\bibitem{Pfeffer}
W. F. Pfeffer,
\emph{The Riemann Approach to Integration: Local Geometric Theory},
Cambridge Tracts Math., Vol. 109,
Cambridge Univ. Press, Cambridge, 1993.

\bibitem{PetersonThompson}
A. Peterson and B. Thompson,
\emph{Henstock--Kurzweil delta and nabla integrals},
J. Math. Anal. Appl.
323 (2006), no. 1, 162--178.
doi:10.1016/j.jmaa.2005.10.025.

\bibitem{PousoRodriguez}
R. L. Pouso and A. Rodr\'iguez,
\emph{A new unification of continuous, discrete and impulsive calculus through Stieltjes derivatives},
Real Anal. Exchange
40 (2015), no. 2, 319--354.

\bibitem{Ronchim}
V. dos S. Ronchim,
\emph{A Study in Set-Theoretic Functional Analysis: Extensions of $c_0(I)$-valued Operators on Linearly Ordered Compacta and Weaker Forms of Normality on Psi-Spaces},
Ph.D. Thesis,
Univ. S\~ao Paulo, 2021.

\bibitem{RonchimTausk} V. dos S. Ronchim and D. V. Tausk, 
\emph{Extension of $c_0(I)$-valued operators on spaces of continuous functions on compact lines}, Studia Mathematica, 268(2023), pp. 259--289.

\end{thebibliography}
\end{document}